\documentclass[a4paper,american,reqno]{article}
\usepackage{geometry}
\usepackage[utf8]{inputenc}
\usepackage[T1]{fontenc}
\usepackage{booktabs}
\usepackage{xspace}
\usepackage[binary-units=true]{siunitx}
\usepackage{etoolbox}
\robustify\bfseries
\usepackage{tabularx}
\usepackage{amssymb}
\usepackage{todonotes}
\usepackage{collectbox}
\usepackage{placeins}
\usepackage{relsize}
\usepackage{algorithm,algpseudocode}
\usepackage{todonotes}
\usepackage{csquotes}
\usepackage{microtype}

\usepackage[colorlinks,
            citecolor=blue,
            urlcolor=blue,
            linkcolor=blue]{hyperref} % should always be the last package
\usepackage{dsfont}
\usepackage{tikz}
\usepackage{pgfplots}
\usepgfplotslibrary{groupplots}
\usepackage{subcaption}

\makeatletter
\patchcmd{\@settitle}{\uppercasenonmath\@title}{\scshape\large}{}{}
\patchcmd{\@setauthors}{\MakeUppercase}{\scshape\normalsize}{}{}
\makeatother
\vfuzz \hfuzz

\usepackage{amsthm}
\newtheorem{theorem}{Theorem}[section]
\newtheorem{lemma}[theorem]{Lemma}
\newtheorem{prop}[theorem]{Proposition}
\newtheorem{corollary}[theorem]{Corollary}

\newtheorem{definition}[theorem]{Definition}
\newtheorem{remark}[theorem]{Remark}
\newtheorem{example}[theorem]{Example}
\newtheorem{assumption}[theorem]{Assumption}

\usepackage{tabularx}

\usepackage{todonotes}
\usepackage{collectbox}
\usepackage{placeins}
\usepackage{relsize}
\usepackage{csquotes}

\usepackage[utf8]{inputenc}
\usepackage{amssymb,amsmath,amsthm,mathtools}
\usepackage{graphicx}
\usepackage{authblk}
\usepackage{mathabx}
\newcommand\R{\mathbb{R}}

\newcommand\Z{\mathbb{Z}}

\newcommand{\support}[1]{\delta^*((A^{#1})^Tv^{#1} \mid \mathcal{Z}_{#1})}
\newcommand{\supportx}[1]{\delta^*((A^{#1})^Tl_{#1}(x) \mid \mathcal{Z}_{#1})}

\begin{document}
\title{$\Gamma$--counterparts for robust nonlinear combinatorial and discrete optimization}
\author[1]{Dennis Adelhütte}
\author[1]{Frauke Liers}
\affil[1]{Friedrich--Alexander--Universität Erlangen--Nürnberg, Department Data Science, Cauerstraße 11, 91058 Erlangen}
\maketitle
\begin{abstract} 
$\Gamma$--uncertainty sets have been introduced for adjusting the degree of conservatism of robust counterparts of (discrete) linear programs. %Instead of aiming for solutions that are optimal regardless of how the uncertainties manifest, the objective is to ensure robustness against $\Gamma$ uncertainties. 
The contribution of this paper is a generalization of this approach to (mixed--integer) nonlinear optimization programs. We focus on the cases in which the uncertainty is linear or concave but also derive formulations for the general case. By applying reformulation techniques that have been established for nonlinear inequalities under uncertainty, we derive equivalent formulations of the robust counterpart that are not subject to uncertainty. The computational tractability depends on the structure of the functions under uncertainty and the geometry of its uncertainty set. We present cases where the robust counterpart of a nonlinear combinatorial program is solvable with a polynomial number of oracle calls for the underlying nominal program. Furthermore, we present robust counterparts for practical examples, namely for (discrete) linear, quadratic and piecewise linear settings. % e.g. for the quadratic assignment problem. We conduct computational studies for the quadratic assignment problem and the vehicle routing problem with general time windows, both under $\Gamma$--uncertainty, to demonstrate the computational tractability in terms of solution quality and running time.
\end{abstract}


%\todo{es ist an mehreren Stellen (zB im Abstract) gesagt dass die
%  Unsicherheit linear oder nichtlinear, konkav... oder sowas ist. Eine lineare
%  Unsicherheit oder eine nichtlineare gibt es nicht. 
%Es ist die Funktion die linear / nichtlinear ... in der Unsicherheit
%ist. Bitte editieren.}

\section{Introduction} \label{sec:Intro} 
In recent years, optimization under uncertainty has gained importance and popularity. When an optimization program is subject to uncertainty, one can aim to solve it before all data are known -- sometimes, this is even mandatory. % Reasons for this uncertainty may include fluctuating input parameters, inaccuracies in measurement or a pressing need to make decisions before all data have been disclosed. 
Two fields of research how to treat said uncertainties are stochastic and robust optimization. In stochastic optimization, one usually requires a sufficient amount of data to estimate or determine the underlying probability distribution. For further information on this, we refer the reader to the monograph \cite{Stochastic}. For robust optimization, one does not require probability distributions. Instead, for modeling the real--life situation, an uncertainty set is pre--determined and one optimizes over the variables while taking the uncertainty set into account simultaneously. Several approaches in robust optimization have been conducted in the past decades, especially in the field of (mixed--integer) linear programming. However, for combinatorial optimization, those are usually not applicable since the underlying program's structure is changed, rendering solution algorithms for the nominal problem not applicable. To circumvent this, Bertsimas and Sim introduced $\Gamma$--uncertainty sets in \cite{IntroGamma} and \cite{PriceRobust} for combinatorial optimization under interval uncertainty in a linear objective. Since recent research has focused on nonlinear robust optimization programming, we extend their approach for combinatorial and discrete programming with nonlinearities.

\paragraph*{Contribution:}
We propose and study a generic framework for mixed--integer nonlinear programs (MINLPs) under uncertainties that generalizes the $\Gamma$--uncertainties for mixed--integer linear programs (MIPs) in \cite{IntroGamma} and \cite{PriceRobust}. We focus on objective uncertainty: On the one hand, we provide reformulations, in particular for the case of nonlinear, concave and linear uncertainty. On the other hand, we show that the programs 
\begin{gather}\label{Form1}
\min_{x \in \mathcal{X}} u^TBg(x)
\end{gather}
underlying an 'assignment structure' and
\begin{gather}\label{Form2}
\min_{x \in \mathcal{X}} u^Tl(x)
\end{gather}
where $l$ is a $0/1$--function\footnote{We call a function $l\colon M \to N$ a $0/1$--function when $l(x) \in \{0,1\}$ for all $x \in M$.} and, in both cases, $u$ is a vector of uncertain parameters in $[\overline{u}, \overline{u} + \Delta u]$ can be solved by a polynomial number of oracle calls (With an oracle, we henceforth mean a black box that can solve the original program without uncertainties). To be more precise, we conduct the following:
\begin{enumerate}
\item Programs with concave uncertainty can be reformulated into programs including support functions and concave conjugations. If the uncertainty is linear, then support functions are sufficient.
%\item If only the objective is subject to (linear) uncertainty and of the form $u^Tl(x)$ with $l$ being a $0/1$--function, then one can solve (and approximate) the problem under uncertainty with a polynomial number of oracle--calls.
%\item Blubb. \todo[inline]{Hier müsste man noch einnähen, ob das andere eine exponential number wird}
\item A black box for solving the program without uncertainty is sufficient for solving programs either~\eqref{Form1} or~\eqref{Form2} under $\Gamma$--uncertainty with intervals. 
\item We propose a new model to handle deadline uncertainty with $\Gamma$--uncertainty sets and show its computational tractability.
\item Our model unifies several applications of $\Gamma$--uncertainty sets in the literature.
\end{enumerate}
In our appendix, we demonstrate the practical applicability of our presented reformulations for the quadratic assignment problem (QAP, \cite{QAPintro}) and a special case of the vehicle routing problem with general time windows (VRPGTW, \cite{VRPSTW}), both under $\Gamma$--uncertainty, in terms of a prototypical numerical study. Uncertainty in the constraints can be handled analogously and is also briefly discussed in the electronic companion. In total, our goal is to present a unifying framework for nonlinear optimization under uncertainty that is of interest for future research, in particular the combination of combinatorial optimization and nonlinear programming under uncertainty.\\
%We foOur proposed model can be reformulated into a MINLP not influenced by uncertainty but where one has to deal with support functions and concave conjugates. Furthermore, if $f_i(x,u^i) = (u^i)^Tl_i(x)$ where $l_i(x)$ is $0/1$--valued and $u^i$ lies in an interval uncertainty, then we show that its robust counterpart can be solved with a polynomial number of oracle calls . 

\paragraph*{Outline:}
The paper is structured as follows. In Section~\ref{sec:Rev}, we briefly revisit the the oracle--polynomial reformulation of the $\Gamma$--counterpart from \cite{IntroGamma} before we introduce the $\Gamma$--counterpart for MINLPs. We motivate our generalization with several applications. In Section~\ref{sec:Non_Linear_Objective_Functions}, we obtain reformulations of our robust counterparts and demonstrate our main results by showing and show that in special cases, oracle--polynomiality holds. In Section~\ref{Sec:Reformulations}, we present various examples and applications of our results with a focus on quadratic problems under uncertainty. Finally, in Section~\ref{Sec:Conclus}, we provide a conclusion and propose some interesting avenues for research. In our electric appendix, we demonstrate the case of uncertain constraints and a prototypical numerical study for the QAP and the VRPGTW under $\Gamma$--uncertainty. %In Section~\ref{Sec:App}, we present computational results of the QAP and the VRPGTW under $\Gamma$--uncertainty. 
% We start by extending the original results for cases where the uncertainty is linear, the objective is separable in the uncertainty and the decision vectors with $0/1$--valued functions. For the rest of Section~\ref{sec:Non_Linear_Objective_Functions}, we apply different reformulations and techniques, mainly based on \cite{IntroGamma} and \cite{Fenchel} to obtain a reformulation of the counterpart for the underlying problem under uncertainty (both in the combinatorial and the general case). We also obtain reformulations for objective functions of the form $u^TBx$ where $u$ is subject to interval uncertainty. We close this section by providing some notes about the complexity class of our counterparts. 
% and state examples for computationally tractable combinatorial optimization problems that fall in this category, e.g. the Quadratic Assignment Problem. 
%Furthermore, we analyze a special case for which the reformulations of Bertismas and Sim are applicable even when the uncertainties are correlated among the variables. This case generalizes binary problems with an assignment structure and we demonstrate this for a single-machine scheduling problem. %For the non-concave case, we present two methods from \cite{Fenchel} and demonstrate that we can generate robust counterparts for optimization problems with a quadratic objective. 

\paragraph*{Literature review:}
A first discussion of a program subject to uncertainty has been conducted by Soyster \cite{Soyster} for column--wise uncertainty of a constraint matrix where the uncertainty set is a convex set. To deal with the resulting over--conservatism, several approaches have been introduced in the literature. In \cite{ben-tal1} and \cite{ben-tal3}, the authors have presented and discussed linear programs under uncertainty from a theoretical and a practical point of view, focusing on convex/interval uncertainties, respectively. The case of reformulating convex/concave rather than linear functions under uncertainty has been addressed in \cite{ben-tal2}.  A framework for treating robust programming with reformulation approaches is \cite{RobustBook}. Another approach for treating uncertainties is the adversarial approach presented in \cite{ApplDec2} that, instead of reformulating, tries to iteratively add scenarios of the uncertainty to the nominal program, finds optimal solutions and checks whether the found solutions are already robust. The approaches are compared in \cite{BertsimasCutting}. For broad overviews of robust optimization in a theoretical and applied sense, we refer to the surveys \cite{survey2011}, \cite{survey2014} and \cite{SurveyAdj}. In the context of combinatorial/discrete programs under uncertainty, we reference \cite{Kouvelis_1997} as the first framework and the surveys \cite{RobCombOpt} and \cite{Kasperski2016} that tackle interval, discrete and convex uncertainties in the objective. $\Gamma$--uncertainties were introduced in \cite{PriceRobust} and applied to combinatorial programs in \cite{IntroGamma}.  
% For the sake of completeness, an overview over the concept of adjustable robustness (that also tackles discrete problems), we refer to \cite{SurveyAdj}. 
%Thus, Bertsimas and Sim introduced $\Gamma$--uncertainties in \cite{IntroGamma} and \cite{PriceRobust}.
Throughout the last two decades, they have been generalized and extended. In \cite{BusingBand} and in \cite{Fischetti2009}, new uncertainty concepts, namely multi--band uncertainty and light robustness, based on $\Gamma$--uncertainties, have been introduced. In \cite{Poss12}, \cite{Poss14} and \cite{POSS2018}, uncertainty sets that generalized $\Gamma$--uncertainty sets were introduced. Furthermore, \cite{Buesing2019} discussed $\Gamma$--uncertainty sets in the context of global robust optimization and \cite{Goerigk2021} introduced locally budgeted uncertainty sets. Finally, $\Gamma$--uncertainty sets have also been applied in a dynamic robust sense, e.g. in \cite{ApplicationBudget}. Optimization under uncertainties including nonlinearities are fairly new in the literature. For an overview over theory, solution approaches and applications, we refer to \cite{NonlinearRO}. The reformulation techniques that our framework is based on have been derived in \cite{Fenchel}.

Finally, for our applications, we refer to the following sources. The QAP under uncertainty was discussed in \cite{QAP2015}, \cite{QAP2012} and \cite{QAM}. The VRPTGW under uncertainty is motivated by the patient transport problem described in \cite{Dummies} and based on the VRPGTW described in \cite{VRPSTW}. To the best of our knowledge, the aforementined $\Gamma$--uncertainty set for piecewise linear objective functions is new in the literature, although piecewise linear functions under uncertainty have been discussed e.g. in \cite{ApplPWL2013} and \cite{Gorissen_2013}.

\section{Our modeling framework} \label{sec:Rev}
\subsection{Revisiting $\Gamma$--uncertainties for binary programs}
Since our focus lies on reformulations of robust counterparts, we revisit a reformulation result of \cite{IntroGamma}. We consider a combinatorial
program with cost vector $\overline{c} \in \R^n$ and feasible set $\mathcal{X} \subseteq \{0,1\}^n$:
\begin{align}\label{Prob:NomComb}
\begin{split}
\min_{x \in \mathcal{X}} & \ \overline{c}^Tx. \\
\end{split}
\end{align} 
%The feasible set $\mathcal{X} \subseteq \{0,1\}^n$ encodes its combinatorial structure. %of
%problem~\eqref{Prob:NomComb}.%, e.g. the spanning trees or the paths
%of a graph.
We assume that the cost coefficients are subject to interval uncertainty, i.e., $c_i \in [\overline{c}_i, \overline{c}_i + \Delta c_i]$ for a given $\Delta c_i \geq 0$ for all $i \in [n]:=\{1, \dots, n\}$. We aim to find solutions that are robust against at most $\Gamma \in [n]$ coefficients deviating from their nominal scenario $\overline{c}_i$: %In this case, we obtain the following robust counterpart
\begin{align}\label{Prob:Gamma_Comb}
\min_{x \in \mathcal{X}} \ \left\{ \overline{c}^Tx + \max_{\mathcal{S} \subseteq [n]: \left| \mathcal{S} \right| \leq \Gamma} \left\{ \sum_{i \in \mathcal{S}} \Delta c_i x_i \right\} \right\}.
\end{align} 
%Assuming that $\mathcal{X}$ is the intersection of a polyhedron with $\{0,1\}^n$, one could apply Proposition~\ref{Prop:Gamma_MIP}. However, one would change the feasible set $\mathcal{X}$ and thus would in general not be able to apply combinatorial algorithms which were applicable for the nominal problem. This can drastically improve the running time (e.g. Shortest-Path-Problem...) of solving the optimization problem.
%\todo[inline]{Dennis: Quelle für Kürzestes-Wege-Probleme mit MIP-Formulierung} 
%Especially in the case of combinatorial optimization problems for
%which the costs are not correlated, this notion makes sense.
In \cite{IntroGamma}, Bertsimas and Sim have shown that the optimal solutions of program~\eqref{Prob:Gamma_Comb} can be found by applying
an optimization oracle of program~\eqref{Prob:NomComb}: % feasible set $\mathcal{X}$ only consists of binary vectors, one can reformulate Problem~\eqref{Prob:RobCounterpartMIP} into a problem where combinatorial algorithms are still applicable:
\begin{prop}(\cite{IntroGamma}, Theorem 3)\label{Prop:Tractable_Gamma_Comb}
  %Consider problem~\eqref{Prob:Gamma_Comb} and
Assume that $\mathcal{X} \subseteq \{0,1\}^n$ and set $\Delta c_{0}:=0$.  Then program~\eqref{Prob:Gamma_Comb} is equivalent to 
\begin{align}\label{Prob:RobustCounterpartComb}
\min_{k \in [n]_0} \left\{ \Gamma \Delta c_k + \min_{x \in \mathcal{X}} \left\{\overline{c}^Tx + \sum_{j \in [n]} \max\{0,\Delta c_j - \Delta c_k\}x_j \right\} \right\}
\end{align}
where $[n]_0 := \{0, 1, \dots, n\}$.
\end{prop}
Proposition~\ref{Prop:Tractable_Gamma_Comb} implies that program~\eqref{Prob:Gamma_Comb} can be solved to optimality in oracle--polynomial time, assuming that an oracle for program~\eqref{Prob:NomComb} is at hand. We note that it is crucial that $\mathcal{X} \subseteq\{0,1\}^n$. If $\mathcal{X} = \{x \in \Z^p \times \R^{n-p}: \ Ax \leq b, \ x \in [r,s] \}$ for $r,s \in \R^n$ and $p \in [n]$, then one can reformulate the robust counterpart as a computationally tractable MIP but has to introduce additional variables while losing the structure of the original program. Thus, the oracle for solving program~\eqref{Prob:Gamma_Comb} is in general not applicable. For details, we refer the reader to the proofs in \cite{IntroGamma}.

\subsection{Introduction of our model and applications} \label{Sec:General}
In this subsection, we extend  program~\eqref{Prob:Gamma_Comb} to MINLPs that are subject to uncertainty in the
objective. %We note that the case of constraints subject to uncertainty can be handled analogously and can be found in Appendix A.
%In addition, by including several examples from a range of areas, we show that the new robust counterpart is beneficial in various situations, as it leads to new formulations and less conservative solutions, when compared to the standard counterpart from Proposition~\ref{Prop:Tractable_Gamma_Comb}. Finally, we apply the counterpart to nonlinear discrete optimization problems. Since the case of uncertainty in the constraints can be
%tackled analogously, we refer, for the rest of this paper, to Appendix
%A for this case.\\
We consider the program
\begin{align}\label{Prob:GeneralNomOpt}
\begin{split}
\inf_{x \in \mathcal{X}} & \sum_{i \in [m]} \ \overline{f}_i(x),
\end{split} 
\end{align}
where $\overline{f}_i \colon \R^n \to \R$ is an arbitrary but fixed function for every $i \in [m]$ and $\mathcal{X} \subseteq \Z^p \times \R^{n-p}$ where $p \in [n]_0$. We assume that every function $\overline{f}_i$ is 'contaminated' by an uncertainty set $\mathcal{U}_i \subseteq \R^{L_i}$, i.e., we define %extend $\overline{f}_i$ to 
%\begin{gather*}
$f_i \colon \R^n \times \mathcal{U}_i \to \R$
%\end{gather*} 
with $f_i(x,\overline{u}^i):= \overline{f}_i(x)$ for a nominal scenario $\overline{u}^i \in \mathcal{U}_i$ and $L_i$ is the dimension of the uncertain parameter $u^i \in \mathcal{U}_i$. We focus on the uncorrelated case, i.e., we assume that uncertainties of different functions $f_i$ are uncorrelated, i.e., the uncertainty set of $\sum_{i \in [m]} f_i(x, u^i)$ is $\mathcal{U} := \bigtimes_{i \in [m]} \mathcal{U}_i$. This ensures that the different uncertainty sets $\mathcal{U}_i$ have no influence on each other. Hence the robust counterpart of program~\eqref{Prob:GeneralNomOpt} without 'restricting' $\mathcal{U}$ further is
\begin{align}\label{Prob:GeneralStrictOpt}
\begin{split}
\inf_{x \in \mathcal{X}} & \sum_{i \in [m]} \ \sup_{u^i \in \mathcal{U}_i} f_i(x,u^i).
\end{split} 
\end{align} 
We say that, if $\mathcal{U}^i$ is convex and $f_i(x,\cdot): \mathcal{U}^i \to \R$ is concave for every $x \in \mathcal{X}$, then the uncertainty is concave. Furthermore, if there exists a function $l_i\colon \mathcal{X} \to \R^{L_i}$ with $f_i(x,u^i) = (u^i)^Tl_i(x)$ for every $x \in \mathcal{X}, u^i \in \mathcal{U}^i$, we call the uncertainty linear.
%An overview of methods for tackling
%problem~\eqref{Prob:GeneralStrictOpt} can be found e.g. in
%\cite{NonlinearRO}.  \todo[inline]{Hier nochmal erwähnen?}
%Future research could consider generalizations of problem~\eqref{Prob:GeneralStrictOpt}
%by taking correlated uncertainties into
%account %of each
%%function $f_i$
%%are correlated, i.e.,
%so that the objective under uncertainty reads $\max_{u \in
% \mathcal{U}} \sum_{i \in [m]} f_i(x,u^i)$.\\ %Here, we aim at
%generalizing the flexible $\Gamma$-approach for uncorrelated
%interval uncertainties and leave the correlated case
%subject to future research. \\
%In the linear case in Section 2, one attempted to be robust against
%$\Gamma$ changing coefficients.
\\
To reduce over--conservatism, we aim to be robust against at most $\Gamma$ functions deviating from their nominal scenario and obtain the following:
\begin{alignat}{2}\label{Prob:GeneralGammaCounterpart}
\inf_{x \in \mathcal{X}} & \left\{ \sup_{\mathcal{S} \subseteq [m]: |\mathcal{S}| \leq \Gamma} \left\{ \sum_{i \in \mathcal{S}} \sup_{u^i \in \mathcal{U}_i} f_i(x,u^i) + \sum_{i \in [m] \setminus \mathcal{S}} f_i(x, \overline{u}^i) \right\} \right\}.
\end{alignat}
Program~\eqref{Prob:GeneralGammaCounterpart} is henceforth referred to as the $\Gamma$-counterpart of program~\eqref{Prob:GeneralNomOpt}. Naturally, it is more general than~\eqref{Prob:Gamma_Comb}. To demonstrate that this generalization is natural, we consider linear uncertainties: Assume that $f_i$ is subject to linear uncertainty for every $i \in [m]$ and that $\mathcal{U}^i$ is convex and compact. Then there exists $w^i \in \mathcal{U}^i$, such that $\sup_{u^i \in \mathcal{U}^i} f_i(x,u^i) = f_i(x,w^i) \in \R$. Then we obtain
\begin{alignat*}{2}
%\inf_{x \in \mathcal{X}} & \left\{ \sup_{\mathcal{S} \subseteq [m]: |\mathcal{S}| \leq \Gamma} \left\{ \sum_{i \in \mathcal{S}} \sup_{u^i \in \mathcal{U}_i} f_i(x,u^i) + \sum_{i \in [m] \setminus \mathcal{S}} f_i(x,\overline{u}^i) \right\} \right\} \\
\eqref{Prob:GeneralGammaCounterpart}
&= \inf_{x \in \mathcal{X}}  \left\{ \sup_{\mathcal{S} \subseteq [m]: |\mathcal{S}| \leq \Gamma} \left\{ \sum_{i \in \mathcal{S}} f_i(x,w^i) + \sum_{i \in [m] \setminus \mathcal{S}} f_i(x, \overline{u}^i) \right\} \right\} \\
&= \inf_{x \in \mathcal{X}}  \left\{ \sup_{\mathcal{S} \subseteq [m]: |\mathcal{S}| \leq \Gamma} \left\{ \sum_{i \in \mathcal{S}} f_i(x,w^i) - f_i(x,\overline{u}^i) + \sum_{i \in [m]} f_i(x, \overline{u}^i) \right\} \right\} \\
&= \inf_{x \in \mathcal{X}}  \left\{ \sum_{i \in [m]} f_i(x, \overline{u}^i) + \sup_{\mathcal{S} \subseteq [m]: |\mathcal{S}| \leq \Gamma} \left\{ \sum_{i \in \mathcal{S}} f_i(x,w^i - \overline{u}^i)\right\} \right\} \\
&= \inf_{x \in \mathcal{X}}  \left\{ \sum_{i \in [m]} f_i(x, \overline{u}^i) + \sup_{\mathcal{S} \subseteq [m]: |\mathcal{S}| \leq \Gamma} \left\{ \sum_{i \in \mathcal{S}} \sup_{z^i \in \mathcal{U}_i - \overline{u}^i} f_i(x,z^i)\right\} \right\}.
\end{alignat*}
In particular, if $\mathcal{U}_i = [\overline{u}_i, \overline{u}_i + \Delta u_i] \subseteq \R_{\geq 0}$ and $f_i(x,u_i) = u_ix_i \geq 0$ for all $u_i \in \mathcal{U}_i$, then
\begin{gather*}
\sup_{z_i \in \mathcal{U}_i - \overline{u}_i} f_i(x,z_i) = \Delta u_i x_i
\end{gather*}
and one obtains program~\eqref{Prob:Gamma_Comb}. \\
Before we continue with our discussion for handling the $\Gamma$--counterpart in Section~\ref{sec:Non_Linear_Objective_Functions}, we present some application examples.
\paragraph{Single--machine scheduling under uncertainty}
This application was firstly discussed under uncertainty in \cite{RobustScheduling} and \cite{Tadayon2015}. A set of $m$ jobs $\mathcal{J}$ must be scheduled on a single machine. The machine requires a processing time $\overline{p}_j \in \R_{\geq 0}$ to finish job $j \in \mathcal{J}$ without preemption. The completion time of job $j$ that depends on the schedule $x$ and on the processing time $\overline{p}:=(\overline{p}_j)_{j \in \mathcal{J}}$ is denoted by $C_j(x,\overline{p})$. With $\mathcal{X}$, we denote the set of feasible schedules (the binary variable $x_{i,j}$ indicates whether job $j$ is scheduled at position $i$):
\begin{gather*}
\mathcal{X} := \left\{x \in \{0,1\}^{[m] \times |\mathcal{J}|}: \sum_{i \in [m]} x_{i,j} = 1 \ \forall j \in \mathcal{J}, \ \sum_{j \in \mathcal{J}} x_{i,j} = 1 \ \forall i \in [m] \right\}.
\end{gather*}
Assuming that every job $j$ contributes weight $w_j \in \R_{\geq 0}$
to the objective, we aim to minimize the total completion time:
\begin{gather}\label{Prob:WeightedScheduling}
\min_{x \in \mathcal{X}} \sum_{j \in \mathcal{J}} C_j(x,\overline{p}).
\end{gather} 
Program~\eqref{Prob:WeightedScheduling} is equivalent to
\begin{gather}\label{Prob:UnWeightedScheduling}
\min_{x \in \mathcal{X}} \sum_{j \in \mathcal{J}} \overline{p}_j \sum_{i \in [m]} (m+1-i)x_{i,j}.
\end{gather}
%where
%\begin{gather*}
%x_{i,j} = \begin{cases} 1, & \text{job $j$ is scheduled at position $i$},\\
%0, & \text{otherwise}. \end{cases}
%\end{gather*}
If the processing time $p_j$ is subject to
uncertainty $\mathcal{U}_j=[\overline{p}_j,
 \overline{p}_j + \Delta p_j]$ for each job $j \in \mathcal{J}$, then the $\Gamma$--counterpart~\eqref{Prob:GeneralGammaCounterpart} of program~\eqref{Prob:UnWeightedScheduling} is
% \begin{alignat*}{2}
%\min_{x \in \mathcal{X}} & \left\{ \max_{\mathcal{S} \subseteq \mathcal{J}: |\mathcal{S}| \leq \Gamma} \left\{ \sum_{j \in \mathcal{S}} \max_{p_j \in \mathcal{U}_j} p_j \sum_{i \in [m]} (m+1-j)x_{i,j} + \sum_{j \in \mathcal{J} \setminus \mathcal{S}} \overline{p}_j \sum_{i \in [m]} (m+1-j)x_{i,j} \right\} \right\},
%\end{alignat*}
%which is equivalent to  %due to $\max_{p_j \in \mathcal{U}_j} p_j \sum_{i \in [m]} (m+1-j)x_{i,j} = (\overline{p}_j + \Delta p_j)\sum_{i \in [m]} (m+1-j)x_{i,j}$,  
\begin{gather}\label{Prob:GammaScheduling}
\min_{x \in \mathcal{X}} \left\{ \sum_{j \in \mathcal{J}} \overline{p}_j \sum_{i \in [m]} (m+1-i)x_{i,j} + \max_{\mathcal{S} \subseteq \mathcal{J}: |\mathcal{S}| \leq \Gamma} \left\{ \sum_{j \in \mathcal{S}} \Delta p_j \sum_{i \in [m]} (m+1-i)x_{i,j} \right\} \right\}
\end{gather}
since the uncertainty is linear. \\
%\end{example}
While program~\eqref{Prob:GammaScheduling} looks very similar to
\eqref{Prob:Gamma_Comb}, there is an important difference: The variables $x_{i,j}$ are not multiplied with exactly one coefficient but each uncertain parameter $p_j$ is multiplied with a linear combination of the variables $x_{i,j}$ for fixed $j \in \mathcal{J}$. Thus, the objective of the nominal program~\eqref{Prob:UnWeightedScheduling} is linear in $p$ and
in $x$ but has the form $\min_{x \in \mathcal{X}} u^TBx$ for a real
matrix $B$ instead of $\min_{x \in \mathcal{X}} u^Tx$. Thus, Proposition 2.1 and the original results of \cite{IntroGamma} cannot be applied.
%The next application we present is the quadratic assignment problem (QAP) which was first discussed under uncertainty in \cite{QAP2012} and \cite{QAP2015} \todo{regret robustness?}  
%\begin{example}\label{Ex:QAP}
\paragraph{Quadratic assignment problem under uncertainty}
The QAP %, see for example \cite{QAP2015},
models the process of assigning $n \in \mathbb{N}$ facilities to $n$ locations such that the cost of transporting goods is minimized. %, resulting of the
%amount of material between facilities multiplied with the distance of
%the locations. 
  %The material handling costs between the facilities, given the flow i.e., the amount of goods transported between two facilities and the distance between two locations should be minimized.
%$c_{ij}$ is the flow between the facilities $i$ and $j$, $(i,j) \in [n]^2$ and $d_{rs}$ is the distance between the locations $r$ and $s$, $(r,s) \in [n]^2$.
%The task is to minimize the cost resulting of the transport of the goods:  TO DO!!!!!
%%which is defined as the product of the amount of material which is transported between two facilities and the distance between the corresponding locations.
%A facility is to assigned to exactly one location and each location is supposed to be used exactly once. 
With binary variables $x_{i,r}$, $i,r \in [n],$ that indicate whether facility $i$ is assigned to location $r$,
%\begin{align*}
%x_{i,r} = \begin{cases}
%1, & \text{if facility } i \text{ is located at }r, \\
%0, & \text{otherwise} \\
%\end{cases}, 
%\end{align*}
the feasible set can be modeled as 
\begin{gather*}
\mathcal{X} = \left\{x \in \{0,1\}^{[n]^2}: \sum_{i\in [n]} x_{i,r}= 1 \ \forall r \in [n], \sum_{r \in [n]} x_{i,r}=1 \ \forall i \in [n]\right\}.
\end{gather*}
For each pair of facilities $(i,j) \in [n]^2$, $c_{i,j} \geq 0$  denotes the flow between $i$ and $j$ and  for all pair of locations $(r,s) \in [n]^2$, $d_{r,s} \geq 0$ denotes the distance between $r$ and $s$. Thus, the QAP can be modeled with
\begin{gather*}
\min_{x\in \mathcal{X}} \sum_{(i,j,r,s) \in [n]^4} c_{i,j} d_{r,s}x_{i,r}x_{j,s}.
\end{gather*}  
In \cite{QAP2015}, the authors have assumed that the flow is subject to interval uncertainty. Their goal was to obtain solutions that are robust against at most $\Gamma$ deviations from the nominal scenario: We seek protection against uncertainties in $c_{i,j}$ that are modeled by a perturbation of at most $\Delta c_{i,j} \geq 0$ for all $(i,j)\in [n]^2$, i.e., $c_{i,j} \in \mathcal{U}_{i,j} = [\overline{c}_{i,j}, \overline{c}_{i,j}+ \Delta c_{i,j}]$. The $\Gamma$-counterpart~\eqref{Prob:GeneralGammaCounterpart} is, since the uncertainty is linear, given by:
\begin{align}\label{Prob:QAP_RobCP}
\begin{split}
&\min_{x\in \mathcal{X}} \left\{ \sum_{(i,j,r,s) \in [n]^4} \overline{c}_{i,j}d_{r,s}x_{i,r}x_{j,s} + \max_{\mathcal{S} \subseteq [n]^2:|\mathcal{S}|\le \Gamma} \left\{  \sum_{(i,j)\in \mathcal{S}}  \sum_{r,s\in [n]} \Delta \overline{c}_{i,j} d_{r,s}x_{i,r}x_{j,s} \right\} \right\} .
\end{split}
\end{align} 
%\end{example}
%The next application is a, to the best of our knowledge, new variant under uncertainty for piecewise linear functions.
%\begin{example}\label{Ex:VRP}
\paragraph{Logistics with deadline uncertainties}
Problems occuring in the application of logistics involving deliveries within given due times can often be modeled as combinatorial programs with (non--)linear objective functions, e.g. taxi routing, delivery of goods or patient transport. For all three of these cases, being on time is important for customer satisfaction. At the same time, it is usually not problematic when vehicle arrives too early for a pick--up. \\
For tasks $i \in [m]$, we denote the due time with $b_i \in \R$.
If a job is finished after $b_i$, then penalty costs occur. A
program for (unweighted) penalty costs is 
\begin{align}\label{Prob:BoundNom}
\begin{split}
\inf_{x \in \mathcal{X}} & \ \sum_{i \in [m]} \max\{0,x_i - b_i\}.
\end{split}
\end{align} 
Program~\eqref{Prob:BoundNom} may arise in transportation logistics, for example as a special case of vehicle
routing problems with general time windows, see \cite{VRPSTW}. 
In practice, the due time can be uncertain: We assume that $b_i \in \mathcal{U}_i := [\overline{b}_i - \Delta b_i, \overline{b}_i]$ for some nominal scenario $\overline{b}_i$ and a perturbation $\Delta b_i$. 
To reduce conservatism, the objective is to ensure robustness against $\Gamma$ deviations of the due times, resulting in the following program:
\begin{align}\label{Prob:UncertainBounds}
\begin{split}
\inf_{x \in \mathcal{X}} & \ \left\{ \sup_{\mathcal{S} \subseteq [m]: |\mathcal{S}| \leq \Gamma} \left\{ \sum_{i \in \mathcal{S}} \max\{0,x_i - \overline{b}_i + \Delta b_i\} + \sum_{i \in [m] \setminus \mathcal{S}} \max\{0, x_i - \overline{b}_i\} \right\} \right\}. 
\end{split}
\end{align}
Problem~\eqref{Prob:UncertainBounds} is obtained from $\Gamma$--counterpart~\eqref{Prob:GeneralGammaCounterpart} by setting $f_i(x,b) := \max\{0, x_i - b_i\}$.

\section{Reformulations for programs with uncertain objectives} \label{sec:Non_Linear_Objective_Functions} 
In this section, we present equivalent reformulations for the
$\Gamma$--counterpart introduced in Section~\ref{Sec:General}. Several of the proofs are inspired by those in \cite{IntroGamma}. It turns out that it is possible to obtain first reformulations of $\Gamma$-- counterpart~\eqref{Prob:GeneralGammaCounterpart} without any assumptions on the functions $f_i$ or the uncertainty sets $\mathcal{U}^i$.  % It also turns out that the bottleneck in solving the $\Gamma$--counterpart is solving $\max_{u^i \in \mathcal{U}_i} f_i(x,u^i) - f_i(x,\overline{u}^i)$ for each $i \in [m]$: 
\begin{lemma}\label{Lem:ReModel}
Let $\Gamma \in [m]$. Then $\Gamma$--counterpart~\eqref{Prob:GeneralGammaCounterpart} is equivalent to 
\begin{align}\label{Prob:GeneralReModel}
\begin{split}
\inf_{x,p,\theta} & \ \Gamma \theta + \sum_{i \in [m]} f_i(x,\overline{u}^i) + p_i, \\
\mathrm{s.t.\;} & x \in \mathcal{X}, \\
& p_i + \theta \geq \sup_{u^i \in \mathcal{U}_i} f_i(x,u^i) - f_i(x, \overline{u}^i) \ \forall i \in [m], \\
& p \in \R^m_{\geq 0}, \theta \in \R_{\geq 0}.
\end{split}
\end{align}
\end{lemma}
\begin{proof}
The structure of the proof is similar to the proof of Theorem 3 in \cite{IntroGamma}. With the binary variables
\begin{gather*}
s_i := \begin{cases} 1, \text{ if $i \in \mathcal{S}$,}\\ 
0, \text{ otherwise,} \\
\end{cases} 
\end{gather*}
$i \in [m]$, the inner maximization program of $\Gamma$--counterpart \eqref{Prob:GeneralGammaCounterpart} is equivalent to
\begin{align}\label{Prob:LinearOrdering}
\begin{split}
\sup_{s} & \ \sum_{i \in [m]} f_i(x, \overline{u}^i) + s_i(\sup_{u^i \in \mathcal{U}_i} f_i(x, u^i) - f_i(x, \overline{u}^i)), \\
\text{s.t.} & \ \sum_{i \in [m]} s_i \leq \Gamma, \\
& \ s \in \{0,1\}^m.
\end{split}
\end{align}
%This is a binary program with $\sum_{i \in [m]} f_i(x,
%\overline{u}^i)$ being independent of the decision variable $s$ and a
%knapsack constraint with equal, unit weights and integral
%capacity. Thus,
%we can recycle the proof of
%Proposition~\ref{Prop:Tractable_Gamma_Comb}:
Clearly, program~\eqref{Prob:LinearOrdering} is equivalent to its LP
relaxation. Inserting its dual into $\Gamma$--counterpart~\eqref{Prob:GeneralGammaCounterpart} proves the claim.
%Its dual is
%\begin{alignat*}{2}
%\min_{p, \theta} & \ \Gamma \theta + \sum_{i \in [m]} f_i(x, \overline{u}^i) + \sum_{i \in [m]} p_i, \\
%\text{s.t.} & \ p_i + \theta \geq \max_{u^i \in \mathcal{U}_i} f_i(x,u^i) - f_i(x, \overline{u}^i) \ \forall i \in [m], \\
%& p, \theta \geq 0.
%\end{alignat*} 
\end{proof}
In Lemma~\ref{Lem:ReModel}, to obtain a tractable formulation, it is necessary to reformulate the inequality
%obtain a tractable formula for $\sup_{u^i \in \mathcal{U}_i} f_i(x,u^i) - f_i(x, \overline{u}^i),$ i.e. to solve a maximization problem over the uncertainty set. In the following subsection, we assume that the uncertainty is concave to apply results of \cite{Fenchel}. In \cite{Fenchel}, the authors proposed various ways of reformulating inequality
\begin{gather}\label{Ineq:RobustNonCon}
p_i + \theta \geq \sup_{u^i \in \mathcal{U}_i} f(x,u^i) - f(x, \overline{u}^i)
\end{gather}
for all $i \in [m]$. In \cite{Fenchel}, the authors proposed various approaches, especially for linear/concave uncertainties which will be discussed in the subsequent subsections. We demonstrate one approach for the non--concave case that uses the notion of term--wise parallel vectors for a non--convex quadratic program in Section~\ref{Sec:Reformulations}. For other approaches, we refer to \cite{Fenchel} and \cite{NonlinearRO}. \\

Furthermore, one can reformulate program~\eqref{Prob:GeneralReModel} to obtain a program with feasible set $\mathcal{X}$ and without variables $p$ and $\theta$:
\begin{lemma}\label{Rem:NoFenchel}
If $\Gamma \in [m]$, then $\Gamma$--counterpart~\eqref{Prob:GeneralGammaCounterpart} is equivalent to 
%Theorem~\ref{Theo:GeneralGammaCounterPartWithoutDual} also holds in a more general version: By Lemma~\ref{Lem:ReModel},
%the $\Gamma$--counterpart~\eqref{Prob:GeneralGammaCounterpart} is equivalent to model~\eqref{Prob:GeneralReModel} and by applying the proof of Theorem~\ref{Theo:GeneralGammaCounterPartWithoutDual} to model~\eqref{Prob:GeneralReModel} we obtain
\begin{align}\label{Prob:GammaCombCounterpartWithoutFenchel}
\begin{split}
\inf_{k \in [m]_0} & \left\{ \inf_{x \in \mathcal{X}} \left\{ \Gamma \theta^k(x) + \sum_{i \in [m]} f_i(x,\overline{u}^i) + \sup\{0, \theta^i(x) - \theta^k(x) \} \right\} \right\}, \\
\end{split}
\end{align}
where $\theta^k(x) := \sup_{u^k \in \mathcal{U}_k} f_k(x,u^k) - f_k(x,\overline{u}^k)$ and $\theta^0(x) := 0$. 
%Note that % The only difference between problems~\eqref{Prob:GammaCombCounterpartWithoutFenchel} and \eqref{Prob:GeneralGammaCounterpartWithoutDual} is that one can not apply Fenchel's Inequality. However, due to definition, $\theta^k(x) \geq 0$. In other words: 
%the only difference between \eqref{Prob:GeneralGammaCounterpartWithoutDual} and \eqref{Prob:GammaCombCounterpartWithoutFenchel} is that in the general case, Proposition~\ref{Prop:GeneralPrinciple} could not be applied (which results in the lack of variables $v^i$) and $\max_{u^k \in \mathcal{U}_k} f_k(x,u^k) - f_k(x, \overline{u}^k)$ still needs to be solved. We include an example of this case in Subsection~\ref{Subsec:Non-Concave}.
\end{lemma}
\begin{proof}
Since $\Gamma \in [m]$, $\Gamma$--counterpart~\eqref{Prob:GeneralGammaCounterpart} is equivalent to \eqref{Prob:GeneralReModel}. Since for all $i \in [m]$, $p_i$ only occurs in exactly one inequality, we obtain
\begin{gather}\label{Eq:optimal_p}
p^*_i = \sup \{0, \sup_{u^i \in \mathcal{U}_i} f_i(x^*,u^i) - f_i(x^*, \overline{u}^i) )- \theta^* \} \ \forall i \in [m]
\end{gather}
for an optimal solution $(x^*,p^*,\theta^*)$ of~\eqref{Prob:GeneralReModel}. Inserting equation~\eqref{Eq:optimal_p} into the objective function of \eqref{Prob:GeneralReModel} results in 
\begin{align}\label{Eq:piecewiselinearlem}
 \Gamma \theta + \sum_{i \in [m]} f_i(x,\overline{u}^i) + \sup \{0, \sup_{u^i \in \mathcal{U}_i} f_i(x,u^i) - f_i(x, \overline{u}^i) ) - \theta \}.
\end{align}
Since $\eqref{Eq:piecewiselinearlem}$ is convex and piecewise linear in $\theta$, either $\theta^*=0$ or %Next, we prove that either $\theta^* = 0$ or there exists $k \in [m]$ such that
%\begin{gather}\label{Eq:OptimalTheta}
$\theta^* = \theta^k(x)$ %\max_{u^k \in \mathcal{U}_k} f_k(x^*,u^k) - f_k(x^*,\overline{u}^k)$
%\end{gather}
for one $k \in [m]$.  %Since $\eqref{Eq:piecewiselinear}$ is piecewise linear and convex in $\theta$, the minimum is furthermore obtained either at $0$ or 
\end{proof}
%As an application, we demonstrate the formulation of Lemma~\ref{Rem:NoFenchel} for Example~\ref{Ex:VRP}. 
\subsection{Concave Uncertainties} \label{subsec:Concave}
%Until now, we have made no assumptions about $f_i$ or
%$\mathcal{U}_i$. 
In this subsection, we will focus on programs that fulfill the following assumptions:
\begin{assumption}\label{Ass:General}
For $\Gamma$--counterpart~\eqref{Prob:GeneralGammaCounterpart} and for all $i \in [m]$ we assume:
\begin{itemize}
\item[(i)] There is a nominal scenario $\overline{u}^i \in \R^{L_i}$, a matrix $A^i \in \R^{L_i \times m_i}$ and a convex set $\mathcal{Z}_i \subseteq \R^{m_i}$, such that $
\mathcal{U}_i = \{\overline{u}^i + A^i \zeta^i\mid \ \zeta^i \in
\mathcal{Z}_i\}$, i.e., $\mathcal{U}_i$ is an affine transformation of a convex set (and thus, convex).
\item[(ii)] The uncertainty of $f_i$ is concave (we recall that this definition implies that $f_i(x,\cdot): \mathcal{U}^i \to \R$ is concave for every $x \in \mathcal{X}$).
\item[(iii)] The nominal scenario $\overline{u}^i$ is contained in the relative interior of $\mathcal{U}_i$.
\end{itemize}
\end{assumption}
%Since convex sets stay convex under affine-linear transformations,
%(i) implies that $\mathcal{U}^i$ is convex.
Convexity of the uncertainty set (i) is a typical assumption in robust optimization, see \cite{RobustBook}. 
%Furthermore, we note that $0 \in \mathcal{Z}_i$ implies $\overline{u}^i \in \mathcal{U}^i$ which is, again, a straight-forward assumption. 
%We note that, iwf $f_i(x,u^i)$ is \emph{convex} in $x$, (i) can be assumed without loss of generality; for more details, see 
%\todo[inline][inline]{Dennis: Ben-Tal nachlesen}
Assumptions (ii) and (iii) are required to apply the techniques from \cite{Fenchel}. To this end, we need some tools of convex analysis:
\begin{definition}
Let $\mathcal{X} \subseteq \R^n$ and $\mathcal{A} \subseteq \R^m$ be a convex set. Let $f \colon \mathcal{X} \times \mathcal{A} \to \R, \ (x,a) \mapsto f(x,a)$ be a function that is concave in $a$ for every $x \in \mathcal{X}$. For an arbitrary, but fixed $x \in \mathcal{X}$, the function
\begin{alignat*}{4}
f_{*}(x, \cdot) \colon & \R^n & \to & \ \R \cup \{-\infty\}, \\
& a & \mapsto & \inf_{y \in \mathcal{A}} \{a^Ty - f(x,y)\}
\end{alignat*}
is called the (partial) concave conjugate with respect to $a$. For a non--empty set $\mathcal{S} \subseteq \R^n$, the function
\begin{alignat*}{4}
\delta^*(\cdot \mid \mathcal{S})  \colon & \R^n & \to & \ \R \cup \{\infty\}, \\
& x & \mapsto & \sup_{y \in \mathcal{S}} y^Tx
\end{alignat*}
is called the support function of $\mathcal{S}$.
\end{definition}
%\begin{remark}
%With our notation, we aim to stay close to the notation of \cite{Fenchel}. For example, $\delta^*(\cdot \mid \mathcal{S})$ is the \emph{convex conjugate} of the indicator function. Since we do not require the exact definition of convex conjugates or indicator functions, we neglect them in this paper. For more details, see \cite{Fenchel} or \cite{ConvexAnalysis}.
%\end{remark}
%\todo[inline][inline]{Dennis: Proper definitions...}
The main result of \cite{Fenchel} is the following reformulation:
\begin{prop}(\cite{Fenchel}, Theorem 2)\label{Prop:GeneralPrinciple}
Under Assumption~\ref{Ass:General}, inequality
\begin{gather*}
\sup_{u^i \in \mathcal{U}_i} f_i(x,u^i) \leq 0
\end{gather*}
is satisfied if and only if there is a vector $v^i \in \R^{L_i}$, such that
\begin{gather}\label{Ineq:GeneralPrinciple}
\left(\overline{u}^i\right)^T v^i + \delta^*((A^i)^Tv^i \mid \mathcal{Z}_i) - f_{i,*}(x,v^i) \leq 0.
\end{gather}
\end{prop}
%\begin{proof}
%See Theorem 2 in \cite{Fenchel}.
%\end{proof}
%Proposition~\ref{Prop:GeneralPrinciple} implies that we can verify
%whether a given $x$ satisfies a given inequality under uncertainty by
%solving a deterministic inequality without uncertainty. Furthermore,
We note that convexity of $f_i$ in $x$ implies that the left-hand side of inequality~\eqref{Ineq:GeneralPrinciple} is also convex, since $\delta^*$ is convex in $v^i$ and $f_{i,*}$ is concave in $(x,v^i)$, see \cite{Fenchel}.\\ By applying Proposition~\ref{Prop:GeneralPrinciple} to
program~\eqref{Prob:GeneralReModel}, we obtain an equivalent reformulation of $\Gamma$--counterpart~\eqref{Prob:GeneralGammaCounterpart}: %which is deterministic: 
\begin{corollary}\label{Prop:FenchelGammaCounterpart}
Let $\Gamma \in [m]$. Under Assumption~\ref{Ass:General}, $\Gamma$--counterpart~\eqref{Prob:GeneralGammaCounterpart} is equivalent to %Problem~\eqref{Prob:GeneralReModel} 
\begin{align}\label{Prob:FenchelGammaCounterpart}
\begin{split}
\inf_{x,p,\theta,v^1, \dots, v^m} & \ \Gamma \theta + \sum_{i \in [m]} f_i(x,\overline{u}^i) + p_i \\
\mathrm{s.t.}\; & x \in \mathcal{X} \\
& p_i + \theta \geq (\overline{u}^i)^Tv^i + \support{i} - f_{i,*}(x,v^i) - f_i(x, \overline{u}^i) \ \forall i \in [m] \\
& p, \theta \geq 0.
\end{split}
\end{align}
%%Furthermore, if $f_i$ is linear in the uncertainty for some $i \in [m]$, i.e., $f_i(x,u^i)={(u^i)}^Tl_i(x)$ for a function $l_i \colon \R^n \to \R^{L_i}$, the inequality constraints of problem~\eqref{Prob:FenchelGammaCounterpart} are equivalent to 
%%\begin{align}\label{Prob:FenchelGammaCounterpartLinear}
%%\begin{split}
%%%\min_{x,p,\theta} & \ \Gamma \theta + \sum_{i \in [m]} f_i(x,\overline{u}^i) + p_i, \\
%%%\text{s.t.} & x \in \mathcal{X}, \\
%% p_i + \theta \geq  \supportx{i}\\
%%%& p, \theta \geq 0.
%%\end{split}
%%\end{align}
%%for each $i \in [m]$.
\end{corollary}
\begin{proof}
Since $\Gamma$ is integral, Lemma~\ref{Lem:ReModel} holds. Proposition~\ref{Prop:GeneralPrinciple} implies that for each $i
\in [m]$, 
\begin{gather*}
p_i + \theta \geq ({\overline{u}^i})^Tv^i + \support{i} - f_{i,*}(x,v^i) - f_i(x, \overline{u}^i)
\end{gather*}
is satisfied for $v \in \R^{L_i}$ if and only if 
%\begin{gather*}
$p_i + \theta \geq \sup_{u^i \in \mathcal{U}_i} f_i(x,u^i) - f_i(x, \overline{u}^i).$
%\end{gather*}
%which proves the first part of the claim. In the case where $f_i$ is linear in $u^i$, we have
%\begin{gather*}
%f_{i,*}(x,v^i) \neq -\infty \Leftrightarrow v^i = l_i(x),
%\end{gather*} 
%which implies $f_{i,*}(x,l_i(x)) = 0$ \cite{Fenchel}. Since inequality~\eqref{Ineq:GeneralPrinciple} is not fulfilled for $f_{i,*}(x,v^i) = - \infty$, $v^i = l_i(x)$ holds. Inserting this into the inequality constraints of problem~\eqref{Prob:FenchelGammaCounterpart} results in the claim, since $f_i(x, \overline{u}^i) = (\overline{u}^i)^Tl_i(x)$.
%\begin{align*}
%%\min_{x,p,\theta} & \ \Gamma \theta + \sum_{i \in [m]} f_i(x,\overline{u}^i) + p_i, \\
%%\text{s.t.} & x \in \mathcal{X}, \\
%p_i + \theta \geq (\overline{u}^i)^Tl_i(x) + \supportx{i} - f_i(x,\overline{u}^i) \ \forall i \in [m]. 
%%& p, \theta \geq 0.
%\end{align*} 
%Since $f_i(x, \overline{u}^i) = (\overline{u}^i)^Tl_i(x)$, the claim follows.
\end{proof}
%\todo{Remark einfädeln als Korollar.}
%\todo{Bemerkung machen, wie sich viele konkave Settings zu linear umformulieren lassen.}
In general, program~\eqref{Prob:FenchelGammaCounterpart} is not compuationally tractable. The support function and the concave conjugate with respect to $v^i$ are optimization programs themselves that depend on a new decision variable $v^i$. In \cite{Fenchel}, the authors derived finite reformulations for various uncertainty sets $\mathcal{U}_i$, including geometries like ellipsoids, polyhedra, cones, boxes, Minkowski sums or their intersections and uncertainty sets that are described by various functions, e.g. convex functions or separable functions and. Their findings can also be applied here. For details, we refer to Tables 1, 2 and 3 in \cite{Fenchel}. 

\subsection{Linear Uncertainties} \label{subsec:Linear}
%Until now, we have made no assumptions about $f_i$ or
%$\mathcal{U}_i$. 
During the last decades, research has been focused on linear uncertainties and they are well--studied. Thus, in this subsection, we will show how one can deal with linear uncertainties in the context of MINLPs under uncertainty, noting that many combinatorial programs are dealing with linear uncertainty (as we will also demonstrate in Section~\ref{Sec:Reformulations}):
\begin{assumption}\label{Ass:GeneralLinear}
For $\Gamma$--counterpart~\eqref{Prob:GeneralGammaCounterpart} and for all $i \in [m]$, Assumption~\ref{Ass:General} holds with the following modification:
\begin{itemize}
%\item[(i)] There is a nominal scenario $\overline{u}^i \in \R^{L_i}$, a matrix $A^i \in \R^{L_i \times m_i}$ and a convex set $\mathcal{Z}_i \subseteq \R^{m_i}$, such that $
%\mathcal{U}_i = \{\overline{u}^i + A^i \zeta^i\mid \ \zeta^i \in
%\mathcal{Z}_i\}$, i.e., $\mathcal{U}_i$ is an affine transformation of a convex set (and thus, convex).
\item[$\text{(ii)}^*$] The uncertainty is linear, i.e., there exists a function $l_i \colon \mathcal{X} \to \R^{L_i}$ such that 
\begin{gather*}
f_i(x,u^i) = (u^i)^Tl_i(x)\  \forall u^i \in \mathcal{U}_i, \ x \in \mathcal{X}.
\end{gather*}
%\item[(iii)] The nominal scenario $\overline{u}^i$ is contained in the relative interior of $\mathcal{U}_i$.
\end{itemize}
\end{assumption}
%Since convex sets stay convex under affine-linear transformations,
%(i) implies that $\mathcal{U}^i$ is convex.
We start by formulating Corollary~\ref{Prop:FenchelGammaCounterpart} under Assumption~\ref{Ass:GeneralLinear}:
\begin{corollary}\label{Prop:FenchelGammaCounterpartLinear}
Let $\Gamma \in [m]$. Under Assumption~\ref{Ass:GeneralLinear}, $\Gamma$--counterpart~\eqref{Prob:GeneralGammaCounterpart} is equivalent to %Problem~\eqref{Prob:GeneralReModel} 
\begin{align}\label{Prob:FenchelGammaCounterpartLin}
\begin{split}
\inf_{x,p,\theta} & \ \Gamma \theta + \sum_{i \in [m]} (\overline{u}^i)^Tl_i(x) + p_i, \\
\mathrm{s.t.}\; & x \in \mathcal{X}, \\
& p_i + \theta \geq  \supportx{i} \ \forall i \in [m], \\
& p, \theta \geq 0.
\end{split}
\end{align}
\end{corollary}
\begin{proof}
Since $\Gamma$ is integral, Corollary~\ref{Prop:FenchelGammaCounterpart} holds. Since the uncertainty is linear, we have
\begin{gather*}
f_{i,*}(x,v^i) \neq -\infty \Leftrightarrow v^i = l_i(x),
\end{gather*} 
see \cite{Fenchel}. Since inequality~\eqref{Ineq:GeneralPrinciple} is naturally not fulfilled for $f_{i,*}(x,v^i) = - \infty$, $v^i = l_i(x)$ holds. Inserting this into program~\eqref{Prob:FenchelGammaCounterpart} proves the claim, since $f_i(x, \overline{u}^i) = (\overline{u}^i)^Tl_i(x)$.
%\begin{align*}
%%\min_{x,p,\theta} & \ \Gamma \theta + \sum_{i \in [m]} f_i(x,\overline{u}^i) + p_i, \\
%%\text{s.t.} & x \in \mathcal{X}, \\
%p_i + \theta \geq (\overline{u}^i)^Tl_i(x) + \supportx{i} - f_i(x,\overline{u}^i) \ \forall i \in [m]. 
%%& p, \theta \geq 0.
%\end{align*} 
%Since $f_i(x, \overline{u}^i) = (\overline{u}^i)^Tl_i(x)$, the claim follows.
\end{proof}
 
\begin{remark}
We note that naturally, similar to $\Gamma$--counterpart~\eqref{Prob:GeneralGammaCounterpart} being a generalization of program~\eqref{Prob:Gamma_Comb}, Corollary~\ref{Prop:FenchelGammaCounterpartLinear} is a generalization of Theorem~1 in \cite{IntroGamma}.
\end{remark}

For combinatorial optimization under interval uncertainty, it is usually essential to obtain a tractable reformulation for which the feasible set is not altered, as oracles for program~\eqref{Prob:GeneralNomOpt} can then be used to solve the resp. $\Gamma$--counterpart. In the case of linear uncertainty, this can be achieved by adding the additional assumption of $l_i$, $i \in [m]$, being non--negative on $\mathcal{X}$ and altering $\mathcal{X}$:

\begin{theorem}\label{Cor:LinInt}
Let $\Gamma \in [m]$ and consider the $\Gamma$--counterpart~\eqref{Prob:GeneralGammaCounterpart} under interval uncertainty $\mathcal{U}_k = [\overline{u}_k, \overline{u}_k + \Delta u_k]$ for some $\overline{u}_k, \Delta u_k \in \R_{\geq 0}^{L_k}$ for all $k \in [m]$. If the uncertainty is linear with $f(x,u_k) = u_k^Tl_k(x)$ and $l_k(x) \geq 0$ holds for all $x \in \mathcal{X}$ and all $k \in [m]$, then program~\eqref{Prob:GeneralGammaCounterpart} is equivalent to   %, $\theta^k(x) = \Delta u_k l_k(x)$, resulting in %%problem~\eqref{Prob:GeneralGammaCounterpart} is equivalent to 
\begin{align}\label{Prob:CombinatorialGammaCounterPartWithoutDualLinear}
\begin{split}
\inf_{k \in [m]_0} & \left\{ \inf_{\mathcal{Q} \subseteq [m]} \left\{ \inf_{x \in \mathcal{X_\mathcal{Q}}} \left\{ \Gamma \Delta u_k^T l_k(x) + \sum_{i \in [m]} \overline{u}_i^Tl_i(x) + \sum_{q \in \mathcal{Q}} \Delta u_q^T l_q(x) - \Delta u_k^T l_k(x) \right\} \right\} \right\},
\end{split}
\end{align}
with $\Delta u_0 := l_0(x) := 0$ and 
\begin{gather*}
\mathcal{X}_\mathcal{Q} := \{x \in \mathcal{X}: \Delta u_q^T l_q(x) - \Delta u_k^T l_k(x) \geq 0 \ \forall q \in \mathcal{Q}, \ \Delta u_q^T l_q(x) - \Delta u_k^T l_k(x) \leq 0 \ \forall q \in [m] \setminus \mathcal{Q}\}.
\end{gather*}
%  If the uncertainty is 1-dimensional and $l_k$ is binary for all $k \in [m]$, i.e., $u_k \in \R$ and $l_k(x) \in \{0,1\}$ for all $x \in \mathcal{X}$, \eqref{Prob:GeneralGammaCounterpart} is equivalent to
%\begin{align}\label{Prob:CombinatorialGammaCounterPartWithoutDualLinearBinary}
%\begin{split}
%\min_{k \in [m]_0} & \Gamma \Delta u_k + \min_{\substack{x \in \mathcal{X},\\ l_k(x) = 1}} \sum_{i \in [m]} (\overline{u}_i + \max \{0, \Delta u_i - \Delta u_k\})l_i(x),
%\end{split}
%\end{align}
%i.e., the $\Gamma$--counterpart is equivalent to the minimum of $m+1$ nominal problems and an additional constraint.
\end{theorem}
%\todo[inline]{Auch hier: Können wir nicht über exponentiell viele Probleme gehen?}
\begin{proof}
Let $e^j$ be the vector of only ones in $\R^j$. Since $l_k$ is non--negative, $\Gamma$--counterpart~\eqref{Prob:GeneralGammaCounterpart} does not change when one replaces $\mathcal{U}_k$ with $\mathcal{U}^\varepsilon_k := [\overline{u}_k - \varepsilon e^{L_k}, \overline{u}_k + \Delta u_k]$ for any $\varepsilon > 0$ and Assumption~\ref{Ass:GeneralLinear} holds. Thus, $\Gamma$--counterpart~\eqref{Prob:GeneralGammaCounterpart} is equivalent to \eqref{Prob:FenchelGammaCounterpart}. Analogously to the proof of Lemma~\ref{Rem:NoFenchel}, one can show that
\begin{gather}\label{Eq:lin_optimal_p}
p^*_i = \max \{0, \supportx{i} - \theta^* \}
\end{gather}
%for an optimal solution $(x^*,p^*,\theta^*,v^{1,*}, \dots, v^{m,*})$ for all $i \in [m]$ and that of~\eqref{Prob:FenchelGammaCounterpart}. Inserting equation~\eqref{Eq:optimal_p} into the objective function of \eqref{Prob:FenchelGammaCounterpart} results in 
%\begin{align}\label{Eq:piecewiselinear}
% \Gamma \theta + \sum_{i \in [m]} f_i(x,\overline{u}^i) + \max \{0, (\overline{u}^i)^Tv^i + \support{i}-f_{i,*}(x,v) - f_i(x, \overline{u}^i) - \theta \}.
%\end{align}
%Since $\eqref{Eq:piecewiselinear}$ is piecewise linear and convex in $\theta$, 
and that $\theta^*$ equals $0$ or there exists $k \in [m]$ such that
\begin{gather}\label{Eq:lin_OptimalTheta}
\theta^* = \supportx{k} = \Delta u_k^Tl_k(x)
\end{gather}
for optimal $p^*$ and $\theta^*$. Note that the last equation in \eqref{Eq:lin_OptimalTheta} holds since $l_k$ is non--negative.
 Thus, $\Gamma$--counterpart~\eqref{Prob:GeneralGammaCounterpart} is equivalent to 
\begin{alignat*}{2}
\inf_{k \in [m]_0} & \left\{ \inf_{\substack{x \in \mathcal{X}}} \left\{ \Gamma u_k^Tl_k(x) + \sum_{i \in [m]} (\overline{u}^i)^Tl_i(x) + \max \{0, \Delta u_i^Tl_i(x) - \Delta u_k^Tl_k(x)\} \right\} \right\}.
%& \ \ \mathrm{s.t.} \ \theta^i(x) \geq 0 \ \forall i \in [m]. 
\end{alignat*}
This proves the claim since $\mathcal{Q}$ encodes which maximization terms are non--negative.
\end{proof}
Theorem~\ref{Cor:LinInt} states that the $\Gamma$--counterpart under said theorem's assumptions can be 'almost' solved by an optimization oracle for solving~\eqref{Prob:GeneralNomOpt}, assuming that one can extend the oracle from $\mathcal{X}$ to $\mathcal{X}_\mathcal{Q}$. However, the number of calls is in $O(m2^m)$, i.e., exponential in the number of uncertain functions (even if $m$ is part of the input). To the best of our knowledge, there is no 'black box' that can go from $\mathcal{X}$ to $\mathcal{X}_\mathcal{Q}$ in the nonlinear combinatorial context which would certainly be an interesting research avenue. Thus, to be able to solve the $\Gamma$--counterpart with an optimization oracle as of now, one requires further assumptions. However, in Section~\ref{Sec:Reformulations}, we show that this is possible for programs 'underlying an assignment structure'. As a tool, we require the following result for models of the form $\min_{x \in \mathcal{X}} u^TBg(x)$ underlying said structure. 

\begin{theorem}\label{Theo:CorrDoNotCare}
Let $B  \in \R_{\geq 0}^{m \times (m n)}$ be a block diagonal matrix where each block consists of a single row and let $g \colon \R^r \to \R_{\geq 0}^{mn}$ be an arbitrary nonnegative function. Consider program 
\begin{gather}\label{Prob:nom_LP_under_unc}
\inf_{x \in \mathcal{X}} u^TBg(x),
\end{gather}
under interval uncertainty, i.e., $u_i \in \mathcal{U}^i := [\overline{u}_i, \overline{u}_i + \Delta u_i] \subseteq \R_{\geq 0}$ for all $i \in [m]$. Furthermore, we assume that
\begin{gather}\label{Set:Assignment}
\mathcal{X} \subseteq \{x \in \R^r\colon \ \forall i \in [m] \exists ! j \in [n]: \ g_{i,j}(x) \text{ is not constant $0$}\}.
\end{gather} 
Let $\Gamma \in [m]$. Then the $\Gamma$--counterparts of program~\eqref{Prob:nom_LP_under_unc} and program
\begin{gather}
\label{Prob:nom_LP_under_unc_not_cor}
\inf_{x \in \mathcal{X}} y^Tg(x),
\end{gather}
under interval uncertainty $y_{(i,j)}\in \mathcal{Y}^{(i,j)} := [\overline{u}_i B_{i,(i,j)}, (\overline{u}_i + \Delta u_i) B_{i,(i,j)}]$ for all $(i,j) \in [m] \times [n]$ are equivalent.
\end{theorem}
\begin{proof}
We begin by introducing some notation: We set $B_{.,(0,0)} := \Delta y_{0,0}:= \overline{y}_{0,0}:=0$ (note that this is a slight abuse of notation and we mean the zero vector or the number $0$, depending on the dimension), $\Delta y:= (\Delta y_{a,b})_{(a,b) \in [m] \times [n]} := (\Delta u_a B_{a,(a,b)})_{(a,b) \in [m] \times [n]}$,  $\overline{y}:=(\overline{y}_{(a,b)})_{(a,b) \in [m] \times [n]} = (\overline{u}_a B_{a,(a,b)})_{(a,b) \in [m] \times [n]}$ and with $B_{k,.}$, we denote the $k$--th row of $B$. By applying Theorem~\ref{Cor:LinInt}, we obtain that the $\Gamma$--counterpart of program~\eqref{Prob:nom_LP_under_unc} is equivalent to
\begin{align}\label{Model:GammaCounterPartCor}
\inf_{k \in [m]_0} & \left\{ \inf_{\mathcal{Q} \subseteq [m]} \left\{ \inf_{x \in \mathcal{X_\mathcal{Q}}} \left\{ \Gamma \Delta u_k B_{k,.}g(x) + \sum_{i \in [m]} \overline{u}_iB_{i,.}g(x) + \sum_{q \in \mathcal{Q}} \Delta u_q B_{q,.}g(x) - \Delta u_k B_{k,.}g(x) \right\} \right\} \right\}.
\end{align}
Since $B$ is a block--diagonal matrix where each block contains exactly one row, one obtains $B_{i,(r,s)} = 0$ for $i \neq r$. Furthermore, since for all $i \in [m]$, $g_{i,j}(x) \neq 0$ for exactly one $j \in [n]$ (which will we be denoted by $j(i)$) we obtain 
\begin{gather*}
B_{q,.}g(x) = \sum_{i \in [m], j \in [n]} B_{q,(i,j)} g_{i,j}(x) = \sum_{i \in [m]} B_{q,(i,j(i))} g_{i,j(i)}(x) = B_{q,(q,j(q))} g_{q,j(q)}(x).
\end{gather*}
%Thus, for fixed $q \in \mathcal{Q}$, we obtain
%\begin{alignat}{2}\label{Eq:Help3}
%%\Delta u_q B_{q,.}g(x) - \Delta u_k B_{k,.}g(x) \geq 0 &\Longleftrightarrow\\
%\Delta u_q B_{q,(q,j(q))} g_{q,j(q)}(x) \geq  \Delta u_k B_{k,(k,j(k))} g_{k,j(k)}(x). &
%\end{alignat}
%we can further reformulate model~\eqref{Model:GammaCounterPartCor}:
%\begin{align}\label{Model_GammaCounterPartCorII}
%\begin{split}
%& \inf_{(k,l) \in [m] \times [n] \cup \{(0,0)\}}  \inf_{x \in \mathcal{X}} \ \Gamma \Delta u_k B_{k,(k,l)} g_{(k,l)}(x) + \overline{u}^TBg(x) \\
%& \sum_{(i,j) \in [m] \times [n]}   \max\{0, \Delta u_i B_{i,(i,j)} g_{(i,j)}(x) - \Delta u_k B_{k,(k,l)} g_{(k,l)}(x)\}.
%\end{split}
%\end{align}
%Note that the case of $k = 0$ has only been changed by notation to $(k,l) = (0,0)$.
The $\Gamma$--counterpart of program~\eqref{Prob:nom_LP_under_unc_not_cor} is equivalent to (again by applying Theorem~\ref{Cor:LinInt} -- note that $l(x) = g(x)$ here and we replace $\mathcal{Q}$ by $\mathcal{Y}$ for the sake of notation):
\begin{align}\label{Model_GammaCounterPartUnCor}
\inf_{(a,b) \in [m] \times [n] \cup \{(0,0)\}} & \left\{ \inf_{\mathcal{Y} \subseteq [m] \times [n]} \left\{ \inf_{x \in \mathcal{X_\mathcal{Y}}} \left\{ \Gamma \Delta y_{(a,b)}g_{(a,b)}(x) +F_{\mathcal{Y},a,b}(x) \right\}  \right\} \right\}
\end{align}
with
\begin{align}\label{Eq:Help}
\begin{split}
F_{\mathcal{Y},a,b}(x) &:= \sum_{(i,j) \in [m] \times [n]} \overline{y}_{(i,j)}g_{(i,j)}x + \sum_{(q,p) \in \mathcal{Y}} \Delta y_{(q,p)}g_{(q,p)}(x) - \Delta y_{(a,b)}g_{(a,b)}(x).
\end{split}
\end{align}
In the following, we show that one can reduce the number of subproblems of program~\eqref{Model:GammaCounterPartCor} from $(m\cdot n+1) \cdot 2^{mn}$ to $(m+1)\cdot 2^m$ and that the resulting program is exactly~\eqref{Model_GammaCounterPartUnCor}. On the one hand, if $b \neq b(a)$, then $g_{(a,b)}(x) = 0$ for all $x \in \mathcal{X}$ by assumption and equation~\eqref{Eq:Help} results in
\begin{align}\label{Eq:Help2}
F_{\mathcal{Y},a,b}(x) = \sum_{(i,j) \in [m] \times [n]} \overline{y}_{(i,j)}g_{(i,j)}(x) + \sum_{(q,p) \in \mathcal{Y}} \Delta y_{(q,p)}g_{(q,p)}(x) = F_{\mathcal{Y},0,0}(x).
\end{align}
Thus, instead of $b \in [n]$, we can fix $b = b(a)$ in program~\eqref{Model_GammaCounterPartUnCor} and only have $m+1$ 'outer problems'. On the other hand, for each $\mathcal{Y} \subseteq [m] \times [n]$, $x \in \mathcal{X}_{\mathcal{Y}}$ only holds if
\begin{align*}
\Delta y_{q,p}g_{q,p}(x) \geq \Delta y_{a,b(a)}g_{a,b(a)}(x) > 0 \ \forall (q,p) \in \mathcal{Y}.
%% \Delta y_{q,p}g_{q,p}(x) \leq \Delta y_{a,b(a)}g_{a,b(a)}(x) \ %%%\forall (q,p) \in ([m] \times [n]) \setminus \mathcal{Y}.
\end{align*}
However, $g_{q,p}(x)$ is not equal to $0$ for some $x$ only if $p=p(q)$. Thus, $\mathcal{X}_\mathcal{Y} = \emptyset$ if $(q,p) \in \mathcal{Y}$ for some $p \neq p(q)$. Thus, program~\eqref{Model_GammaCounterPartUnCor} is equivalent to
\begin{align}\label{Model_GammaCounterPartUnCor2}
\inf_{a \in [m]_0} & \left\{ \inf_{\overline{\mathcal{Y}} \subseteq [m]} \left\{ \inf_{x \in \mathcal{X}_{\overline{\mathcal{Y}}}} \left\{\Gamma \Delta y_{(a,b(a))}g_{(a,b(a))}(x) + F_{\overline{\mathcal{Y}},a}(x) \right\} \right\} \right\}
\end{align}
where 
\begin{align*}
\mathcal{X}_{\overline{\mathcal{Y}}} := \{ x \in \mathcal{X}\colon & \Delta y_{q,p(q)}g_{q,p(q)}(x) \geq \Delta y_{a,b(a)}g_{a,b(a)}(x) \ \forall q \in \overline{\mathcal{Y}}, \\
& \Delta y_{q,p(q)}g_{q,p(q)}(x) \leq \Delta y_{a,b(a)}g_{a,b(a)}(x) \ \forall q \in [m]  \setminus \overline{\mathcal{Y}} \},
\end{align*}
$b(0) := 0$ and
\begin{align*}
F_{\overline{\mathcal{Y}},a}(x) &:=  \sum_{(i,j) \in [m] \times [n]} \overline{y}_{(i,j)}g_{(i,j)}(x) + \sum_{(q,p) \in \overline{\mathcal{Y}}} \Delta y_{(q,p(q))}g_{(q,p(q))}(x) - \Delta y_{(a,b(a))}g_{(a,b(a))}(x).
\end{align*}
By inserting the definition of $\overline{y}$ and $\Delta y$ into program~\eqref{Model_GammaCounterPartUnCor2} and by replacing all indices, one obtains program~\eqref{Model:GammaCounterPartCor}.
%However, since  Now, we assume that $b=a(b)$
%Note that the final equation holds since $\mathcal{X}_{\mathcal{Y}} = \mathcal{X}$ (the right--hand side of inequality~\eqref{Eq:Help3} is $0$ and the left--hand side is non--negative). Furthermore
\end{proof}
Thus, if $g(x)$ is a $0/1$--function, then one can apply Theorem~\ref{Theo:CorrDoNotCare} to obtain an equivalent $\Gamma$--counterpart where the uncertainty is linear and no matrix $B$ is involved. The following corollary demonstrates this for $g(x) = x$:
\begin{corollary}\label{Cor:LPsAssigment}
Consider program \begin{gather}\label{Prob:nom_LP_under_unc:ASSIGN}
\min_{x \in \mathcal{X}} u^TBx
\end{gather}
as in the setting of Theorem~\ref{Theo:CorrDoNotCare} and assume that $
\mathcal{X} \subseteq \{x \in \{0,1\}^{[m] \times [n]}: \sum_{j \in [n]} x_{i,j} = 1 \ \forall i \in [m]\}$.
Let $\Gamma \in [m]$. Then the $\Gamma$--counterpart of \eqref{Prob:nom_LP_under_unc:ASSIGN} is equivalent to
\begin{align*}
\min_{(k,l) \in [m] \times [n] \cup \{(0,0)\}} \left\{ \Gamma \Delta u_k B_{k,(k,l)}  +
\min_{x \in \mathcal{X}} \left\{ \overline{u}^TBx + \sum_{(i,j) \in [m] \times [n]} F_{i,j,k,l}(x) \right\} \right\}
\end{align*}
where $B_{0,.} := \Delta u_0 := 0$ (the first one being a row vector of zeros and the latter being the number $0$) and 
\begin{gather*}
F_{i,j,k,l}(x):=\max\{0, \Delta u_i B_{i,(i,j)} - \Delta u_k B_{k,(k,l)}\}x_{i,j}.
\end{gather*}
\end{corollary}
%\begin{align*} %\label{Prob:RobustCounterpartSuperComb}
%\begin{split}
%\min_{\substack{(k,l) \in [m] \times [n] \\ \lor k = l = 0 }} \left\{\Gamma \Delta u^T B_{.,(k,l)} + 
% \min_{x \in \mathcal{X}}  \overline{u}^TBx + \sum_{(i,j) \in [m] \times [n]} \max \left\{0,\Delta u^T\left( B_{.,(i,j)} - B_{.,(k,l)}\right)\right\}x_{i,j} \right\}.
%\end{split} 
%\end{align*}
%$\min_{x Assume that for all $i \in [m]$, there is a function $l_i\colon \R^n \to \R^{L_i} $, such that $f_i(x,u^i) = (u^i)^Tl_i(x)$, i.e., $f_i$ is linear in the uncertainty and in the decision. Assume that the uncertainty set is $\mathcal{U}^i := [\overline{u}^\overline{u}^i + \Delta u^i
\begin{proof}
This follows from Theorem~\ref{Theo:CorrDoNotCare} and Proposition~\ref{Prop:Tractable_Gamma_Comb} by setting $g(x) = x$.
\end{proof}

\paragraph*{The special case of $0/1$--functions}
Before we apply our theory in Section~\ref{Sec:Reformulations}, we discuss one more case, namely the $\Gamma$--counterpart~\eqref{Prob:GeneralGammaCounterpart} under linear one--dimensional interval uncertainty with $0/1$--functions:
\begin{assumption}\label{Ass:GeneralLinear01}
For the $\Gamma$--counterpart~\eqref{Prob:GeneralGammaCounterpart} and for all $i \in [m]$ we assume:
\begin{itemize}
\item[(i)] The uncertainty set $\mathcal{U}_i$ is a 1--dimensional interval, i.e., $\mathcal{U}_i = [\overline{u}_i, \overline{u}_i + \Delta u_i] \subseteq \R_{> 0}$ and $\Delta u_i > 0$.
\item[(ii)] There is a $0/1$--function $l_i\colon \mathcal{X} \to \{0,1\}$ such that $f_i(x,u_i) = u_il_i(x)$ for all $u_i \in \mathcal{U}_i$.
\end{itemize}
\end{assumption}
%\begin{align}\label{Prob:Non-LinearBin}
%\min_{x \in \mathcal{X}} u^T l(x)
%\end{align}
%where $u \in \R^n$ is uncertain under interval uncertainty and $l(x) \in \{0,1\}^m$ for all $x \in \mathcal{X}$ and $\mathcal{X} \subseteq \{0,1\}^n$. This is a special case of model~\eqref{Prob:GeneralNomOpt} which is still non-linear since there are no assumptions for $l$ (expect that on $\mathcal{X}$, its image is a subset of $\{0,1\}^m$). 
Although Assumption~\ref{Ass:GeneralLinear01} seems restrictive, it covers many combinatorial programs under uncertainty, e.g. the quadratic knapsack problem or the quadratic matching problem. By applying Proposition~\ref{Prop:Tractable_Gamma_Comb}, we obtain the following:
\begin{theorem}\label{Theo:PseudoLin}
Let $\Gamma \in [m]$ and assume that Assumption~\ref{Ass:GeneralLinear01} holds. Then $\Gamma$--counterpart~\eqref{Prob:GeneralGammaCounterpart} is equivalent to
\begin{gather}\label{Prog:PseudoLin}
\inf_{k \in [m]_0} \left\{ \Gamma \Delta u_k + \inf_{x \in \mathcal{X}} \left\{\overline{u}^Tl(x) + \sum_{j \in [m]} \max\{0,\Delta u_j - \Delta u_k\}l_j(x) \right\} \right\}.
\end{gather}
\end{theorem}
\begin{proof}
Under Assumption~\ref{Ass:GeneralLinear01}, program~\eqref{Prob:GeneralNomOpt} is equivalent to $\inf_{(x,y) \in \mathcal{X}_y} u^Ty$ with $\mathcal{X}_y := \mathcal{X} \times l(\mathcal{X}) \subseteq \R^n \times \{0,1\}^{m}$. Then Proposition~\ref{Prop:Tractable_Gamma_Comb} implies that the modified program's $\Gamma$--counterpart is equivalent to
\begin{gather*}
\inf_{k \in [m]_0} \left\{ \Gamma \Delta u_k + \inf_{x \in \mathcal{X}_y} \left\{\overline{u}^Ty + \sum_{j \in [m]} \max\{0,\Delta u_j - \Delta u_k\}y_j \right\} \right\}
\end{gather*}
where $\Delta u_0 := 0$. Since $y_j = l_j(x)$, the claim follows.
\end{proof}
%Naturally, since we applied Bertsimas' and Sim's original result, assuming that $l(x) \in \{0,1\}^m$ is binary is essential. Therefore, the previous theorem does not cover Examples~\ref{Ex:Scheduling} and \ref{Ex:QAP}. However, it does cover cases where $l_i(x)$ is a product of binary variables for each $i \in [m]$ which is the case for example, quadratic combinatorial optimization problems with objective \begin{gather*}
%\sum_{i \in [n]} p_i x_i + \sum_{i \in [n], j \in [n] \setminus \{i\}} p_{i,j}x_{i}x_{j}
%\end{gather*}
%for binary variables $x_i$ and uncertain coefficients. We demonstrate that in Section~\ref{Sec:Reformulations}.
Theorem~\ref{Theo:PseudoLin} demonstrates that one can solve the $\Gamma$--counterpart with an optimization oracle of program~\eqref{Prob:GeneralNomOpt}. This result only implies that additionally, one can reduce the number of oracle calls one has to solve and can determine $\alpha$-approximations (for $\alpha \geq 1$), if program~\eqref{Prob:GeneralNomOpt} is $\alpha$--approximable\footnote{The following definition is not formal and is usually applied for combinatorial programs: Assume that $f^* \in$~$(-\infty,\infty)$ is the optimal value of program~\eqref{Prob:GeneralNomOpt}. Then program~\eqref{Prob:GeneralNomOpt} is called $\alpha$--approximable when there exists a real number $\alpha \geq 1$ and an algorithm $\texttt{ALG}$ with input $(f,\mathcal{X})$, output $\tilde{x}$, the inequality $\alpha f^* \geq f(\tilde{x})$ holds for every instance $(f,\mathcal{X})$ and the running time of algorithm $\texttt{ALG}$ is polynomial in the encoding length of $(f, \mathcal{X})$.}. The proofs are both heavily inspired by the resp. proofs in \cite{IntroGamma} and \cite{Lee2014}:
\begin{theorem}\label{Theo:Approx}
Let $\Gamma \in [m]$. If Assumption~\ref{Ass:GeneralLinear01} holds and program~\eqref{Prob:GeneralNomOpt} is $\alpha$-approximable, then \eqref{Prob:GeneralGammaCounterpart} is $\alpha$--approximable. 
%equivalent to
%\begin{gather*}
%\min_{k \in \mathcal{L}} \Gamma \Delta u_k + \min_{x \in \mathcal{X}} f(x,\overline{u}) + \sum_{i \in [k]} \left(\Delta u_i - \Delta u_k \right)l_i(x)
%\end{gather*}
%for $\mathcal{L} := \{\Gamma + 1, \dots, \Gamma + \gamma, m + 1\}$ with $\gamma$ being the biggest odd integer smaller than $(m + 1) - \Gamma$ and $\Delta_{m+1}:=0$.
%\end{theorem}
\end{theorem}
\begin{proof}
For $k \in [m]_0$, we denote the objective of the $k$--th inner problem of program~\eqref{Prog:PseudoLin} with $G^k(x)$, i.e.,
\begin{gather*}
G^k(x) := \sum_{j \in [m]} (\overline{u}_j + \max\{0,\Delta u_j - \Delta u_k\})l_j(x). 
\end{gather*}
Naturally, one can $\alpha$--approximate program $\inf_{x\in \mathcal{X}} G^k(x)$ for each $k \in [m]_0$ by assumption. Let $x^k$ be the output of the given approximation algorithm with objective value $z^k$ and $Z^*$ be the optimal value of program~\eqref{Prog:PseudoLin}, which is equivalent to $\Gamma$--counterpart~\eqref{Prob:GeneralGammaCounterpart} since Assumption~\ref{Ass:GeneralLinear01} holds. Then we obtain
\begin{alignat*}{2}
Z^* & \leq \left( \sum_{i \in [m]} l_i(x^k)\overline{u}_i \right) +
\sup_{\mathcal{S} \subseteq [m]: |\mathcal{S}| \leq \Gamma} \sum_{i \in \mathcal{S}}  \Delta u_i l_i(x^k) \\
    & = \left( \sum_{i \in [m]} l_i(x^k)\overline{u}_i \right) + \inf_{\theta \geq 0} \sum_{j \in [m]} \max \{0, \Delta u_i l_i(x^k) -  \theta\} + \Gamma \theta \\
    & \overset{l_i(x^k) \in \{0,1\}}{=} \left(\sum_{i \in [m]} l_i(x^k)\overline{u}_i \right) + \inf_{\theta \geq 0} \sum_{j \in [m]} \max\{0, \Delta u_i  -  \theta \}l_i(x^k) + \Gamma \theta \\
    & \leq \Gamma \Delta u_k + \sum_{i \in [m]} (\overline{u}_i + \max \{0, \Delta u_i  -  \Delta u_k\})l_i(x^k)   \\
    & = \Gamma \Delta u_k + G^k(x^k) \\
    & \leq \alpha(z_k^* - \Gamma \Delta u_k) + \Gamma \Delta u_k \\
    & \overset{\alpha \geq 1}{\leq} \alpha z_k^* \\
    & = \alpha Z^*
\end{alignat*}
which proves the claim since it is sufficient to apply the given $\alpha$--approximation to $\inf_{x\in \mathcal{X}} G^k(x)$ for every $k \in [m]_0$ and to solve $\min_{k \in [m]_0} z_k$.
\end{proof}
\begin{theorem}\label{Theo:RedNumber}
Let $\Gamma \in [m]$ and assume that $\Delta u_1 \geq \Delta u_2 \geq \dots \geq \Delta u_m \geq 0$. If Assumption~\ref{Ass:GeneralLinear01} holds, then the $\Gamma$--counterpart~\eqref{Prob:GeneralGammaCounterpart} is equivalent to
\begin{gather*}
\inf_{k \in \mathcal{L}} \left\{ \Gamma \Delta u_k + \inf_{x \in \mathcal{X}} \left\{ f(x,\overline{u}) + \sum_{i \in [k]} \left(\Delta u_i - \Delta u_k \right)l_i(x) \right\} \right\}
\end{gather*}
for $\mathcal{L} := \{\Gamma + 1, \dots, \Gamma + \gamma, m + 1\}$ with $\gamma$ being the largest odd integer smaller than $(m + 1) - \Gamma$ and $\Delta_{m+1}:=0$. Furthermore, if the optimal value of the $k$--th inner program is smaller than $\Gamma \Delta u_l$ for $l \in \mathcal{L}$, one can replace $\mathcal{L}$ with $\mathcal{L}^* := \{k \in \mathcal{L}: \ k > l\}$.
\end{theorem}
\begin{proof}
Since Assumption~\ref{Ass:GeneralLinear01} holds, this statement is a consequence of Theorem 1 in \cite{Lee2014} by introducing binary variables $y_i$, $i \in [m]$, with $y_i = l_i(x)$ (as in the proof of Theorem~\ref{Theo:PseudoLin}).
\end{proof}
In particular, Theorem~\ref{Theo:RedNumber} implies that, instead of solving $m+1$ nominal programs, one only needs to solve $\lceil \frac{m - \Gamma}{2} \rceil + 1$ subproblems instead of $m+1$, as in the case of Theorem~\ref{Theo:PseudoLin}. \cite{Lee2014} demonstrates that this significantly reduces the number of subproblems one needs to solve for the linear case. %We apply Theorem~\ref{Theo:RedNumber} in our numerical study in Section~\ref{Sec:App}. \\
\\
Before we conclude this section, for the sake of completeness, we note the following: 
\begin{remark}\label{Rem:Concave}
%\begin{enumerate}
%\item[i)] In the setting of Theorem~\ref{Theo:CorrDoNotCare}, one can also show that for the case of $k = 0$, exactly one subproblem needs to be solved, since (in program~\eqref{Model:GammaCounterPartCor}) $x \in \mathcal{X}_\mathcal{Q}$ only holds if and only if
%\begin{alignat*}{2}
%\Delta u_qB_{q,(q,j(q))}g_{q,j(q)}(x)&\geq 0 \ \forall q \in \mathcal{Q},\\
%\Delta u_qB_{q,(q,j(q))}g_{q,j(q)}(x)&\leq 0 \ \forall q \in [m] \setminus \mathcal{Q}
%%% \Delta y_{q,p}g_{q,p}(x) \leq \Delta y_{a,b(a)}g_{a,b(a)}(x) \ %%%\forall (q,p) \in ([m] \times [n]) \setminus \mathcal{Y}.
%\end{alignat*}
%which can only hold for $[m] = \mathcal{Q}$. Since it was not important for our theory, we only mention it here but for the reduction of the number of subproblems, this observation can be crucial. 
With respect to concave uncertainties under Assumption~\ref{Ass:General}, one can show that $\Gamma$--counterpart~\eqref{Prob:GeneralGammaCounterpart} is equivalent to
\begin{alignat}{2}\label{Reform:Concave}
\min_{k \in [m]_0} & \left\{ \min_{\substack{x \in \mathcal{X}, v^1 \in \R^{L_1}, \\  \dots, v^m \in \R^{L_m}}} \left\{ \Gamma \theta^k(x,v^k) + \sum_{i \in [m]} f_i(x,\overline{u}^i) + \max \{0, \theta^i(x,v^i)- \theta^k(x,v^k) \} \right\} \right\}
\end{alignat}
where
\begin{gather*}
\theta^k(x,v^k) := (\overline{u}^k)^Tv^k + \support{k} - f_k(x,\overline{u}^k) - f_{k,*}(x,v^k)  \ \forall k \in [m].
\end{gather*}
In this context, however, it seems unlikely that one would use this formulation over the one given in Corollary~\ref{Prop:FenchelGammaCounterpart} since the objective is very different from the one of program~\eqref{Prob:GeneralNomOpt}, the feasible set was altered, rendering oracles not applicable.
\end{remark}

\section{Practical examples}\label{Sec:Reformulations}
In this section, we demonstrate some examples of the reformulations of Section~\ref{sec:Rev}. We note that some are maximization programs for which our theory naturally applies as well.
\paragraph*{Linear programs under uncertainty}
Here, we cover the case of $\min_{x \in \mathcal{X}} u^TBx$. If $B$ is the unit matrix and $u$ is subject to interval uncertainty, we obtain the original setting of Bertsimas and Sim. We have already shown with Corollary~\ref{Cor:LPsAssigment} that there is a case where one can 'shift' $B$ into the uncertainty set and that one can solve the $\Gamma$--counterpart with a polynomial number of oracle calls. This is a generalization of a result in \cite{RobustScheduling} which has been applied to the single--machine scheduling problem under uncertainty:
\begin{example}\label{Ex:SchedulingGeneral}
We recall that in Section~\ref{sec:Rev}, we considered an instance of the program %problems of the form
\begin{align}\label{Prob:NomScheduling}
\begin{split}
\min_{x} & \ \sum_{i,j \in \mathcal{I}, \mathcal{J}} u_i q_j x_{i,j}, \\
\mathrm{s.t.} & \ x \in \mathcal{X} \subseteq \left\{x \in \{0,1\}^{|\mathcal{I}|\cdot|\mathcal{J}|}:
\sum_{j \in \mathcal{J}} x_{i,j} = 1 \ \forall i \in \mathcal{I}\right\}
\end{split}
\end{align}
where $u_i$ is subject to uncertainty $[\overline{u}_i, \overline{u}_i + \Delta u_i]$ in the context of single--machine scheduling. By applying Corollary~\ref{Cor:LPsAssigment}, one can show that the $\Gamma$--counterpart of program~\eqref{Prob:NomScheduling} is equivalent to taking the minimum of
\begin{align*}
\min_{(k,l) \in \mathcal{I} \times \mathcal{J}} \left\{ \Gamma q_k \Delta u_l + \min_{x \in \mathcal{X}} \left\{ \sum_{i,j \in \mathcal{I}, \mathcal{J}} (\overline{u}_i + \max\{0, \Delta u_i - \frac{\Delta u_k q_l}{q_j}\})q_jx_{i,j} \right\} \right\}
\end{align*}
and 
\begin{align*}
\min_{x \in \mathcal{X}} \sum_{i,j \in \mathcal{I}, \mathcal{J}} (\overline{u}_i + \Delta u_i)q_jx_{i,j}.
\end{align*}
In \cite{RobustScheduling}, this has been shown by algebraic means.
%This was established through algebraic manipulations and reformulations which were used to prove
%Proposition~\ref{Prop:Tractable_Gamma_Comb}. We establish the same result through Corollary~\ref{Cor:LPsAssigment}: By setting
%\begin{gather*}
%B:= \begin{pmatrix}
%q_1   \dots q_{|\mathcal{J}|} & 0_{\mathcal{J}} & \dots & 0_{\mathcal{J}} \\
%0_{\mathcal{J}} & q_1   \dots q_{|\mathcal{J}|} & 0_{\mathcal{J}} \dots & 0_{\mathcal{J}} \\
%\vdots & \vdots & \vdots & \vdots \\
%0_{\mathcal{J}} & 0_{\mathcal{J}} & 0_{\mathcal{J}} & q_1   \dots q_{|\mathcal{J}|} 
%\end{pmatrix} \in \R^{|\mathcal{I}| \times |\mathcal{I}| |\mathcal{J}|},
%\end{gather*}
%and inserting this into~\eqref{Prob:RobustCounterpartSuperComb}, we obtain the same reformulation.
\end{example}

\paragraph*{Quadratic programs under uncertainty}
As an application of Corollary~\ref{Cor:LinInt}, we reformulate the QAP under uncertainty given in Section~\ref{sec:Rev}. 
\begin{example}\label{ex:QAP_genau}
We consider the QAP under interval uncertainty
\begin{align}\label{Prob:QAP_Nom}
\begin{split}
\min_{x} & \ \sum_{(i,j,r,s) \in [n]^4} c_{i,j}d_{r,s}x_{i,r}x_{j,s} \\
\mathrm{s.t.} & \ x \in \mathcal{X} = \{x \in \{0,1\}^{[n]^2}: \sum_{i \in [n]} x_{i,r} = 1 \ \forall r \in [n], \ \sum_{r \in [n]} x_{i,r} = 1 \ \forall i \in [n]\}
\end{split}
\end{align}
with uncertain coefficients $c_{i,j} \in \mathcal{U}_{i,j} := [\overline{c}_{i,j}, \overline{c}_{i,j} + \Delta c_{i,j}] \subseteq \R_{\geq 0}$ for all $i,j \in [n]$ and $d_{r,s} \geq 0$ for all $r,s \in [n]$. Let $\Gamma \in [n^2]$.
The objective of ~\eqref{Prob:QAP_Nom} equals $u^TBg(x)$ where $B \in \R^{n^2 \times n^4}$ is a block diagonal matrix with $n^2$ copies of $d^T$, $u = c$ and $g_{(i,j),(r,s)}(x):= x_{i,r}x_{j,s}
$ for $(i,j,r,s) \in [n]^4$. Then Theorem~\ref{Theo:CorrDoNotCare} implies that the $\Gamma$--counterpart of program~\eqref{Prob:QAP_Nom} is equivalent to the $\Gamma$--counterpart of
\begin{gather*}
\min_{x \in \mathcal{X}} y^Tg(x)
\end{gather*}
with $y \in \R^{n^4}$ being subject to interval uncertainty, in particular $y_{i,j,r,s} \in [\overline{c}_{i,j}d_{r,s}, (\overline{c}_{i,j} + \Delta c_{i,j})d_{r,s}]$ for $i,j,r,s \in [n]$. Now, we can apply Theorem~\ref{Theo:PseudoLin} and we obtain the reformulation
%\begin{align}\label{Prob:QAP_Reform_One}
%\min_{(k_1,k_2,k_3,k_4) \in [n]^4 \cup \{(0,0,0,0)\}} & \min_{x \in \mathcal{X}} \Gamma \Delta (\overline{c}_{k_1,k_2} + \Delta c_{k_1,k_2})d_{k_3,k_4} x_{k_1,k_3}x_{k_2,k_4} + \sum_{(i,j) \in [n]^2} \overline{c}_{i,j}d_{r,s}x_{i,r}x_{j,s} + \max \{0, \Delta {c}_{i,j}d_{r,s}x_{i,r}x_{j,s} - \Delta c_{k_1,k_2}d_{k_3,k_4} x_{k_1,k_3}x_{k_2,k_4} \}
%\end{align}
%where $\bar{c}_{0,0} := \Delta c_{0, 0} := d_{0,0} := x_{0,0} = 0$. We assume that $x_{k_1,k_3}x_{k_2,k_4} = 1$ for $(k_1,k_2,k_3,k_4)\in [n]^4$ (otherwise the term would be $0$ and we would be in the case $k_i = 0$ for $i \in [4]$). Then, 
%\begin{alignat*}{2}
%\max \{0, \Delta {c}_{i,j}d_{r,s}x_{i,r}x_{j,s} - \Delta c_{k_1,k_2}d_{k_3,k_4} x_{k_1,k_3}x_{k_2,k_4} \} &= 
%\max \{0, \Delta {c}_{i,j}d_{r,s}x_{i,r}x_{j,s} - \Delta c_{k_1,k_2}d_{k_3,k_4}  \} =  \\
%&=\max \{0, \Delta {c}_{i,j}d_{r,s} - \Delta c_{k_1,k_2}d_{k_3,k_4}\}x_{i,r}x_{j,s},
%\end{alignat*} 
%since all involved numbers are non-negative. 
%
%Thus, problem~\eqref{Prob:QAP_Reform_One} is equivalent to
\begin{align}\label{Prob:QAP_Reform_Two}
\min_{\substack{(k_1,k_2,k_3,k_4) \in [n]^4 \\ \cup \{(0,0,0,0)\}}} & \left\{ \Gamma (\Delta c_{k_1,k_2})d_{k_3,k_4} + \min_{x \in \mathcal{X}} \left\{ F_{k_1,k_2,k_3,k_4}(x) \right\} \right\}
\end{align}
where 
\begin{gather*}
F_{k_1,k_2,k_3,k_4}(x) = \sum_{(i,j,r,s) \in [n]^4} (\overline{c}_{i,j}d_{r,s} + \max \{0, \Delta {c}_{i,j} d_{r,s} - \Delta c_{k_1,k_2}d_{k_3,k_4} \})x_{i,r}x_{j,s} 
\end{gather*}
for $(k_1,k_2,k_3,k_4) \in [n]^4$, $F_{0,0,0,0}(x) := \sum_{(i,j,r,s) \in [n]^4}  (\overline{c}_{i,j}d_{r,s} + \Delta {c}_{i,j} d_{r,s})x_{i,r}x_{j,s}$, and $\Delta c_{0,0} := d_{0,0} := 0$.
%\begin{align}\label{Prob:QAP_Reform_Two}
%\min_{(k_1,k_2,k_3,k_4) \in [n]^4 \cup \{(0,0,0,0)\}} & \Gamma (\Delta c_{k_1,k_2})d_{k_3,k_4} + \min_{x \in \mathcal{X}, \ x_{k_1 k_3} x_{k_2 k_4} = 1} \sum_{(i,j,r,s) \in [n]^4} (\overline{c}_{i,j}d_{r,s} + \max \{0, \Delta {c}_{i,j} d_{r,s} - \Delta c_{k_1,k_2}d_{k_3,k_4} \})x_{i,r}x_{j,s}.
%\end{align}
%%%Since $x_{k_1 k_3} x_{k_2 k_4} = 1$ is equivalent to $x_{k_1 k_3} = x_{k_2 k_4}= 1$, facilities $k_1$ and $k_2$ are assigned to locations $k_3$ and $k_4$, respectively. Thus, the variables $x_{k_1 j}$, $x_{k_2 j}$, $x_{i k_3}$ and $x_{i k_4}$ are fixed for all $j \in [n]$ and we can neglect those variables. Thus, we can replace the feasible set with 
%%%\begin{gather*}
%%%\mathcal{X}_{k_1,k_2,k_3,k_4} := \left\{x \in \{0,1\}^{[n]_{\setminus k_1,k_2} \times [n]_{\setminus k_3,k_4}}: \ \sum_{i \in [n]_{\setminus k_1,k_2}} x_{i,r} = 1 \ \forall r \in [n]_{\setminus k_3,k_4}, \ \sum_{r \in [n]_{\setminus k_3,k_4}} x_{i,r} = 1 \ \forall i \in [n]_{\setminus k_1,k_2} \right\}.
%%%%\begin{gather*}
%%%%\tilde{\mathcal{X}} := \left\{ x \in \{0,1\}^{([n]\setminus \{k_1,k_2\}) \times ([n] \setminus \{k_3,k_4\})}: \ \sum_{i \in [n]\setminus \{k_1,k_2\}} x_{i,r} = 1 \ \forall r \in [n] \setminus \{k_3,k_4\}, \ \sum_{r \in [n]\setminus \{k_3,k_4\}} x_{i,r} = 1 \ \forall i \in [n] \setminus \{k_1,k_2\} \right\}. 
%%%%
%%%\end{gather*} 
%The inner problems are infeasible, if exactly one of the equations $k_1 = k_2$ or $k_3 = k_4$ holds (for $k_1,k_2,k_3,k_4 \in [n]$). 
If we assume that the flow and the distance coefficients are symmetrical, i.e., $\Delta c_{i,j} = \Delta c_{j,i}$ and $d_{r,s} = d_{s,r}$ for all $i,j,r,s \in [n]$, then we only have to solve inner problems of the set
\begin{gather*}
\mathcal{M}:=\{(k_1,k_2,k_3,k_4) \in [n]^4: \ k_1 < k_2, k_3 < k_4\} \cup \{(0,0,0,0)\}.
\end{gather*}
Thus, one needs to solve $1 + (\frac{n(n-1)}{2})^2 = \frac{n^4 - n^3}{2} + 1$ QAPs to solve program~\eqref{Prob:QAP_Reform_Two}.
%, in particular
%\begin{align}\label{Prob:QAP_Reform_Three}
%\min_{(k_1,k_2,k_3,k_4) \in \mathcal{M}} & \Gamma (\Delta c_{k_1,k_2})d_{k_3,k_4}  + \min_{x \in \mathcal{X}} \sum_{(i,j,r,s) \in [n]^4} (\overline{c}_{i,j}d_{r,s} + \max \{0, \Delta {c}_{i,j} d_{r,s} - \Delta c_{k_1,k_2} d_{k_3,k_4} \})x_{i,r}x_{j,s}.
%\end{align}
By application of Theorem~\ref{Theo:RedNumber}, we can reduce the number of subproblems to $\lceil \frac{n^4 - n^3}{4}  + \frac12 - \frac{\Gamma}{2} \rceil + 1$. In our electronic companion, we demonstrate how the application of Theorem~\ref{Theo:RedNumber} significantly speeds up the the optimization process.
\end{example}

\begin{example}
We consider the following quadratic combinatorial program under interval uncertainty:
\begin{align}\label{Prob:QuadUnc}
\begin{split}
\min_{x \in \mathcal{X} \subseteq \{0,1\}^n} & \sum_{i \in[n]} \sum_{j\in [i]} p_{i,j} x_ix_j
%%\mathrm{s.t} & \ x \in \mathcal{X} := \{x \in \{0,1\}^n: \ \sum_{j\in [n]} \omega_j x_j \le c\}
%\mathrm{s.t} & \ x \in \mathcal{X} \subseteq \{0,1\}^n
\end{split}
\end{align}
with uncertain coefficients $p_{i,j} \in [\bar{p}_{i,j}, \bar{p}_{i,j} + \Delta p_{i,j}]$. Let $m:= n^2 - \frac{n(n+1)}{2} = \frac{n^2 - n}{2}$ be the number of uncertain coefficients and let $\mathcal{M}:=\{(k,l) \in [n]^2: \ l \leq k\}$. For $\Gamma \in [m]$, its $\Gamma$--counterpart is given by
%\begin{align*}
%\max_{x \in \mathcal{X}} & \left\{ \min_{\mathcal{S} \subseteq \mathcal{M}: |\mathcal{S}| \leq \Gamma} \left\{ \sum_{(i,j) \in \mathcal{S}} \min_{p_{i,j} \in [\bar{p}_{ij} - \Delta p_{i,j}, \bar{p}_{ij}]} p_{i,j} x_ix_j + \sum_{(i,j) \in \mathcal{M} \setminus \mathcal{S}} \bar{p}_{i,j} x_ix_j \right\} \right\}
%\end{align*}
%or, equivalently,
\begin{align}\label{Prob:RobQP}
\min_{x \in \mathcal{X}} & \left\{ \sum_{(i,j) \in \mathcal{M} } \bar{p}_{i,j} x_ix_j + \max_{\mathcal{S} \subseteq \mathcal{M}: |\mathcal{S}| \leq \Gamma} \left\{ \sum_{(i,j) \in \mathcal{S}} \Delta p_{i,j}x_ix_j \right\} \right\}
\end{align}
By applying Theorem~\ref{Theo:PseudoLin}, we obtain that program~\eqref{Prob:RobQP} is equivalent to
\begin{align*}
\min_{(k,l) \in \mathcal{M} \cup \{(0,0)\}} \left\{ \Gamma \Delta p_{k,l} + \min_{x \in \mathcal{X}}   \left\{ \sum_{i \in[n]} \sum_{j\in [i]} (\bar{p}_{ij} + \max \{0, \Delta p_{i,j} - \Delta p_{k,l}\}) x_ix_j  \right\} \right\}
\end{align*}
where $\Delta p_{0,0} := 0$ and we can solve the robust counterpart with $m+1$ calls of an optimization oracle of program~\eqref{Prob:QuadUnc}.

%Linearisierung: Standart
%\begin{align*}
% \min_{x \in \{0,1\}^n, y\in \{0,1\}^{n\cdot n}}& \sum_{i \in[n]} \sum_{j\in [n]} p_{ij} y_{ij}\\
%\text{s.t} &\sum_{j\in n} \omega_j x_j \le c\\
%&y_{ij} \ge x_i +x_j -1\\
%&y_{ij} \le x_i \\
%&y_{ij} \le x_j
%\end{align*}
%Reformuliert dann zu
%\begin{align*}
% \min_{x \in \{0,1\}^n, y\in \{0,1\}^{n\cdot n}}& \sum_{i \in[n]} \sum_{j\in [n]} p_{ij} y_{ij} +\max_{S\subseteqeq [n]\times [n]:|S|\le \Gamma}  \sum_{(i,j) \in S}  \Delta p_{ij} y_{ij}\\
%\text{s.t} &\sum_{j\in n} \omega_j x_j \le c\\
%&y_{ij} \ge x_i +x_j -1\\
%&y_{ij} \le x_i \\
%&y_{ij} \le x_j
%\end{align*}
%
%\begin{align*}
% \min_{x \in \{0,1\}^n, y\in \{0,1\}^{n\cdot n}, z}& \sum_{i \in[n]} \sum_{j\in [n]} p_{ij} y_{ij} +\Gamma z_0 + \sum_{i,j \in [n]} z_{ij}\\
%\text{s.t} &\sum_{j\in n} \omega_j x_j \le c\\
%&y_{ij} \ge x_i +x_j -1\\
%&y_{ij} \le x_i \\
%&y_{ij} \le x_j\\
%& z_{ij}+z_0 \ge \Delta p_{ij} y_{ij}\\
%& z_{ij} \ge 0
%\end{align*}
%\todo[inline][inline]{Da vertraue ich dir mal, dass das passt :P Erstmal klären, ob das mit dem Knapsack wirlkich alles passt.}
\end{example}
To conclude the discussion of quadratic programs under uncertainty, we demonstrate the  approach of applying the notion of term--wise parallel vectors and hidden concavity \cite{TermWise} for interval uncertainty. 
\begin{example}
Let $\mathcal{X}$ be a convex set. We consider the following program:
\begin{align}\label{Prob:NomQuad}
\inf_{x \in \mathcal{X}} \sum_{i \in [m]} (x_i - u_i)^2.
\end{align}
For every $i \in [m]$, $u_i$ is subject to uncertainty $\mathcal{U}_i := [\overline{u}_i, \overline{u}_i + \Delta u_i]$. For $\Gamma \in [m]$, the $\Gamma$--counterpart of program~\eqref{Prob:NomQuad} after applying Lemma~\ref{Lem:ReModel} is given by
\begin{align}\label{Prob:FirstQuad}
\begin{split}
\inf_{x, p, \theta} & \ \Gamma \theta + \sum_{i \in [m]} (x_i - \overline{u}_i)^2 + p_i, \\
\mathrm{s.t.} & \ x \in \mathcal{X}, \\
& \max_{u_i \in \mathcal{U}_i} (x_i - u_i)^2 - (x_i - \overline{u}_i)^2 \leq p_i + \theta \ \forall i \in [m].
\end{split}
\end{align}
The inequalities of program~\eqref{Prob:FirstQuad} are equivalent to
\begin{gather}\label{Ineq:ReformedQuad}
-2x_i u_i + u_i^2 \leq p_i + \theta - 2x_i \overline{u}_i + \overline{u}_i^2 \ \forall u_i \in \mathcal{U}_i
\end{gather}
for each $i \in [m]$. The left--hand side of \eqref{Ineq:ReformedQuad} is not concave in $u_i$ and thus, we cannot apply the reformulations of Subsection~\ref{subsec:Concave}. We set $h(y) := y^2$ for all $y \in \R$  
%\eqref{Ineq:ReformedQuad} is of the form \eqref{Ineq:SecondApproach}: We set
and consider the following formulation of the uncertainty set:
\begin{gather*}
[\overline{u}_i, \overline{u}_i + \Delta u_i] = \{u_i \in \R\colon (-2 \overline{u}_i - \Delta u_i)u_i + u_i^2 \leq -\overline{u}_i^2 - \overline{u}_i \cdot \Delta u_i\}.
\end{gather*}
Finally, define
\begin{gather*}
\alpha := -2\overline{u}_i - \Delta u_i, \ \beta := 1, \ \gamma = -\overline{u}_i^2 - \overline{u}_i \Delta u_i. \
\end{gather*}
Since $\alpha$ and $\beta$ are scalars, they are clearly term--wise parallel \cite{TermWise}. Thus, we obtain that \eqref{Ineq:ReformedQuad} is satisfied for $(x, p, \theta)$ if and only if there exist $v_1, \dots, v_m \in \R$, such that
\begin{align*}
&(1 + v)h_i^*\left(\frac{2x + 2 \overline{u}_iv_i + \Delta u_iv_i}{1+v_i}\right) - \overline{u}_i^2v_i - \overline{u}_iv_i \Delta u_iv_i + 2x_i \overline{u}_i\le p_i + \theta  + \overline{u}_i^2,\\
&1 + v_i \ge 0,\\
&v_i\ge 0,
\end{align*}
see \cite{Fenchel}, Subsection 4.3, for details. Clearly, the second inequality is redundant. Furthermore, since
$h_i^*(z) = \frac{z^2}{4}$, the first inequality is convex in $(x,v_i)$ if it is assumed that $x \geq 0$. If $\mathcal{X} \subseteq \R^n_{\geq \overline{u}}$, then the $\Gamma$--counterpart of program~\eqref{Prob:NomQuad} is a convex optimization problem, although the uncertainty is not concave.
\end{example}

\paragraph*{Deadline uncertainties in a piecewise linear setting}
We conclude our discussion of applications with the deadline uncertainty setting we introduced in Section~\ref{sec:Rev}.
%This concludes our discussion of concavity in the uncertainty. In the next subsection, we propose methods to tackle the general case.

%Looking at \ref{Prob:GeneralReModel}, we try to compute with Fenchel duality a tractable formulation for $p_i+\theta \ge \max_{u^i \in \mathcal{U}_i} f_i(x,u^i) -f_i(x, \bar{u}^i)$. This equals $ p_i+\theta - \max_{u^i \in \mathcal{U}_i} f_i(x,u^i) +f_i(x, \bar{u}^i)\ge 0$. Here it is possible to replace the maximum with a for all condition. We also assume that $f_i(x,u)$ is separable:
%\[-p_i - \theta - f_i(x, \bar{u}^i) + \bar{f}_i(u^i)^Tg_i(x)\le 0\]
%Following approach on we replace $\bar{f}_(u)=b$. With this we obtain:
%\[b_i^T\bar{g}_i(x) \le 0, \qquad \forall (u_i,b_i) \in \bar{U}^i=\{(u,b): u \in \mathcal{U}^i, \bar{f}_i(u^i)=b_i\} \}.\]
%Now we have to obtain $conv(\bar{U}^i)$. However, this depends strongly on the structure of $\mathcal{U}_i$. For a general $conv(\bar{U}_i)$ the resulting robust counterpart can be determined as:
%\begin{align*}
%\min_{x,p,\theta} & \Gamma \theta + \sum_{i \in [m]} f_i(x, \bar{u}^i +p_i\\
%\text{s.t} & x \in \mathcal{X}\\
%& a^{0T}\bar{g}_i(x)+\delta^*(A^T\bar{g}_i(x)|Z_i)\le 0\\
%&p,\theta \ge 0
%\end{align*} 
%%Apart from these approaches, there are cases that enable a formula for $\max_{u^i \in \mathcal{U}_i}
%%f_i(x,\overline{u}^i)$ analytically. We apply this approach to the problem introduced in
%%Example~\ref{Ex:VRP}:
\begin{example}\label{Ex:UncertainBounds}
Consider the $\Gamma$--counterpart of program~\eqref{Prob:BoundNom} as introduced in Section~\ref{sec:Rev}:%Example~\ref{Ex:VRP}:
\begin{align*}
\inf_{x \in \mathcal{X}}  \left\{ \sup_{\mathcal{S} \subseteq [m]: |\mathcal{S}| \leq \Gamma} \left\{ \sum_{i \in \mathcal{S}} \max\{0,x_i - \overline{b}_i + \Delta b_i\} + \sum_{i \in [m] \setminus \mathcal{S}} \max\{0, x_i - \overline{b}_i\} \right\} \right\}.
\end{align*}
An equivalent reformulation, given in Remark~\ref{Rem:NoFenchel} as program~\eqref{Prob:GammaCombCounterpartWithoutFenchel}, is
\begin{align}\label{Prob:PiecewiseLinearFirstStep}
\begin{split}
\inf_{k \in [m]_0} \left\{ \inf_{x \in \mathcal{X}} \left\{  \Gamma \theta^k(x) + 
\sum_{i \in [m]}  \max \{0, x_i - \overline{b}_i, \max\{0,x_i - \overline{b}_i + \Delta b_i\} - \theta^k(x)\} \right\} \right\}
\end{split}
\end{align}
with $\theta^k(x) := \max \{0, x_k - \overline{b}_k + \Delta b_k\} - \max\{0,x_k - \overline{b}_k\}$ and $\theta^0(x) := 0$. Thus, for $k = 0$, it is necessary to solve
$ \min_{x \in \mathcal{X}} \sum_{i \in [m]} \max \{0, x_i - \overline{b}_i + \Delta b_i\}$. For $k > 0$, we distinguish between three cases:
\begin{enumerate}
\item[i)] $x_k \geq \overline{b}_k$, i.e., $\theta^k(x) = \Delta b_k$ and
\begin{gather}\label{Eq:Case1}
\max \{0, x_i - \overline{b}_i, \max \{0,x_i - \overline{b}_i + \Delta b_i\} - \theta^k(x)\} = \max\{0, x_i - \overline{b}_i, x_i - \overline{b}_i + \Delta b_i -\Delta b_k\}.
\end{gather}
\item[ii)] $x_k \in [\overline{b}_k- \Delta b_k, \overline{b}_k]$, i.e., $\theta^k(x)  = x_k - \overline{b}_k + \Delta b_k$ and
%Continuing with \eqref{Eq:ObjBounds} results in 
\begin{align}\label{Eq:Case2}
\begin{split}
&\max \{0, x_i - \overline{b}_i, \max \{0,x_i - \overline{b}_i + \Delta b_i\} - \theta^k(x)\} =\\ & \max\{0, x_i - \overline{b}_i, x_i - \overline{b}_i + \Delta b_i - x_k + \overline{b}_k - \Delta b_k\}.
\end{split}
\end{align}
\item[iii)] $x_k \leq \overline{b}_k - \Delta b_k$, i.e., $\max \{0, x_k - \overline{b}_k\} = \max \{0, x_k - \overline{b}_k + \Delta b_k\}=0$. Thus, $\theta^k(x) = 0$ and we refer to the case of $k = 0$.
\end{enumerate}
For each $k \in [m]$, by applying equations~\eqref{Eq:Case1} and \eqref{Eq:Case2}, we solve 
\begin{align}\label{Prob:SubProbVRPNOTRec}
\begin{split}
\inf_{x \in \mathcal{X}} & \ \sum_{i \in [m]} \max\{0, x_i - \overline{b}_i, x_i - \overline{b}_i + \Delta b_i - \Delta b_k \} + \Gamma \Delta b_k \\
\mathrm{s.t.} & \ x_k \in [\overline{b}_k - \Delta b_k, \overline{b}_k]
\end{split}
\end{align}
and
\begin{align}\label{Prob:SubProbVRPRec}
\begin{split}
\inf_{x \in \mathcal{X}} & \ \sum_{i \in [m]} \max\{0, x_i - \overline{b}_i, x_i - \overline{b}_i + \Delta b_i - x_k + \overline{b}_k -\Delta b_k \} + \Gamma (x_k - \overline{b}_k + \Delta b_k) \\
\mathrm{s.t.} & \ x_k \geq b_k.
\end{split}
\end{align}
%For $k = 0$, we need to solve
%\begin{align*}
%\min_{x \in \mathcal{X}} & \ \sum_{i \in [m]} \max \{0,x_i - \overline{b}_i + \Delta b_i\}.
%\end{align*}
Therefore, in total, we need to solve $2m + 1$ optimization programs with a piecewise linear objective with the addition of one additional hard bound for exactly one variable for $k \in [m]$. In the electric companion, we apply this reformulation to a special case of the vehicle routing problem with general time windows.
%Applying the example of logistics, one can interpret this in the following way: For $k \in [m]$, we solve two optimization problems: we distinct between the nominal bound $\overline{b}_k$ being satisfied and not being satisfied. In both cases, one compares the violations for $x_i$ and $x_k$ and adds $\Gamma$ times the violation of bound $x_k$. For $k = 0$, $\Gamma$ is multiplied and $0$ and vanished but all bounds realize their worst-case scenario. So basically, it is a trade-off between worst-case scenarios and the impact of bound $\overline{b}_k$ being not satisfied.
%We note that (except for $k = 0$), the number of pieces doubles ($2$ instead of $4$). For piecewise linear optimization it is crucial that the number of pieces does not rise exponentially.
\end{example}
\section{Conclusion} \label{Sec:Conclus}
In this paper, we studied $\Gamma$--counterparts of discrete nonlinear
optimization problems under uncertainty in the objective. We established reformulations of $\Gamma$--counterparts by applying reformulations techniques developed in \cite{Fenchel}. Similar to $\Gamma$--uncertainties in \cite{IntroGamma} and \cite{PriceRobust}, our reformulations work for general MINLPs and for combinatorial optimization problems with linear uncertainty when attempting to optimize over the original feasible set $\mathcal{X}$. While those reformulations are not necessarily computationally tractable, we have provided examples where this is indeed the case, namely for linear uncertainties involving $0/1$--functions, programs involving some kind of assignment structure.
Furthermore, we discussed the general case with an application in logistics. %Finally, for our computational study, we showed that our reformulations are applicable in practice and can outperform linearized $\Gamma$--counterparts.
Possible further research for this topic include extensive numerical studies for the derived reformulations that could be based on our prototypical study in the electronic companion. Furthermore, one could also investigate whether the generalizations of \cite{Poss12}, \cite{Poss14} and \cite{POSS2018} to the nonlinear $\Gamma$--counterpart are possible and tractable as well. Finally, one could attempt to investigate cases where one could extend the optimization oracle from $\mathcal{X}$ to $\mathcal{X}_\mathcal{Q}$ as in the setting of Theorem~\ref{Cor:LinInt} or to perform, under further assumptions, more reformulations that reduce the number of the problems from exponential to polynomial im $m$, similar to Theorem~\ref{Theo:CorrDoNotCare}. %Furthermore, the price of robustness was not investigated in this publication as it was in \cite{PriceRobust} which could also be subject to future research.
%Our theoretical results predominantly center on reformulations and generalizations that raise further questions: Furthermore, one could further analyze the effect of coefficients correlated by uncertainty in a linear setting, not only for the $\Gamma$--approach but in general. In \cite{PriceRobust}, MIPs with constraints subject to uncertainty have been discussed. In this case, the authors proposed probabilistic guarantees for a solution $x$ staying feasible when the number of coefficients realizing an uncertainty bigger than the anticipated $\Gamma$. It would be interesting to obtain similar probabilistic guaranteed for MINLPs. %Finally, from a practical point of view, one could conduct numerical studies that compare our counterparts with other models and concepts of robust optimization. %While we outlined possibilities for the QAP already in Section~\ref{Sec:App}, one could use other approaches in the sense of different concepts for robust optimization, e.g. the adversarial approach \cite{ApplDec2} or an adjustable setting that applies decision rules, e.g. \cite{ApplDec1}.

\section*{Funding and Acknowledgements} 
The authors are grateful to Jana Dienstbier for many fruitful discussions, in particular on the quadratic assignment problem under uncertainty.
Research reported in this paper was partially supported by project HealthFaCT under BMBF grant 05M16WEC. It was also funded by the
Deutsche Forschungsgemeinschaft (DFG, German Research Foundation) -- Project--ID 416229255 -- SFB 1411. Furthermore, this paper has received funding from the European Union’s Horizon 2020 research and innovation program under the Marie Skłodowska-Curie grant agreement No 764759.

\bibliographystyle{gOMS}
%\bibliography{gOMSguide}
\bibliography{Bibliography_OR}

\begin{thebibliography}{10}
\newcommand{\noopsort}[1]{}
\newcommand{\printfirst}[2]{#1}
\newcommand{\singleletter}[1]{#1}
\newcommand{\switchargs}[2]{#2#1}
\providecommand{\url}[1]{\normalfont{#1}}
\providecommand{\urlprefix}{Available at }

\bibitem{Dummies}
D. Adelhütte, K. Braun, F. Liers, and S. Tschuppik, \emph{{Minimizing}
  {Delays} of {Patient} {Transports} with {Incomplete} {Information}}, Tech.
  {R}ep., Optimization Online,  2021,
  \urlprefix\url{http://www.optimization-online.org/DB_HTML/2021/02/8242.html}.

\bibitem{ApplPWL2013}
A. Ardestani-Jaafari and E. Delage, \emph{Robust optimization of sums of
  piecewise linear functions with application to inventory problems}, Oper.
  Res. 64 (2016), pp. 474--494,
  \urlprefix\url{https://doi.org/10.1287/opre.2016.1483}.

\bibitem{TermWise}
A. Ben-Tal, D. den  Hertog, and M. Laurent, \emph{Hidden convexity in partially
  separable optimization}, CentER Working Paper Series No. 2011-070  (2011),
  \urlprefix\url{https://ssrn.com/abstract=1865208}.

\bibitem{Fenchel}
A. Ben-Tal, D. den  Hertog, and J.P. Vial, \emph{Deriving robust counterparts
  of nonlinear uncertain inequalities}, Math. Program. 149 (2015), pp.
  265--299,
  \urlprefix\url{https://link.springer.com/article/10.1007/s10107-014-0750-8}.

\bibitem{RobustBook}
A. Ben-Tal, L. El~Ghaoui, and A. Nemirovski, \emph{Robust optimization},
  Princeton Series in Applied Mathematics, Princeton University Press,
  Princeton, NJ, 2009, \urlprefix\url{https://doi.org/10.1515/9781400831050}.

\bibitem{ben-tal2}
A. Ben-Tal and A. Nemirovski, \emph{Robust convex optimization}, Math. Oper.
  Res. 23 (1998), pp. 769--805,
  \urlprefix\url{https://doi.org/10.1287/moor.23.4.769}.

\bibitem{ben-tal1}
A. Ben-Tal and A. Nemirovski, \emph{Robust solutions of uncertain linear
  programs}, Oper. Res. Lett. 25 (1999), pp. 1--13,
  \urlprefix\url{https://doi.org/10.1016/S0167-6377(99)00016-4}.

\bibitem{ben-tal3}
A. Ben-Tal and A. Nemirovski, \emph{Robust solutions of linear programming
  problems contaminated with uncertain data}, Math. Program. 88 (2000), pp.
  411--424, \urlprefix\url{https://doi.org/10.1007/PL00011380}.

\bibitem{survey2011}
D. Bertsimas, D.B. Brown, and C. Caramanis, \emph{Theory and applications of
  robust optimization}, SIAM Rev. 53 (2011), pp. 464--501,
  \urlprefix\url{https://doi.org/10.1137/080734510}.

\bibitem{BertsimasCutting}
D. Bertsimas, I. Dunning, and M. Lubin, \emph{Reformulation versus
  cutting-planes for robust optimization: a computational study}, Comput.
  Manag. Sci. 13 (2016), pp. 195--217,
  \urlprefix\url{https://doi.org/10.1007/s10287-015-0236-z}.

\bibitem{IntroGamma}
D. Bertsimas and M. Sim, \emph{Robust discrete optimization and network flows},
  Math. Program. 98 (2003), pp. 49--71,
  \urlprefix\url{https://doi.org/10.1007/s10107-003-0396-4}, integer
  programming (Pittsburgh, PA, 2002).

\bibitem{PriceRobust}
D. Bertsimas and M. Sim, \emph{The price of robustness}, Oper. Res. 52 (2004),
  pp. 35--53, \urlprefix\url{https://doi.org/10.1287/opre.1030.0065}.

\bibitem{ApplicationBudget}
D. Bertsimas, E. Litvinov, X. Sun, J. Zhao, and T. Zheng, \emph{Adaptive robust
  optimization for the security constrained unit commitment problem}, IEEE
  Transactions on Power Systems 28 (2013), pp. 52--63,
  \urlprefix\url{https://ieeexplore.ieee.org/abstract/document/6248193}.

\bibitem{ApplDec2}
D. Bienstock and N. \"{O}zbay, \emph{Computing robust basestock levels},
  Discrete Optim. 5 (2008), pp. 389--414,
  \urlprefix\url{https://doi.org/10.1016/j.disopt.2006.12.002}.

\bibitem{RobustScheduling}
M. Bougeret, A.A. Pessoa, and M. Poss, \emph{Robust scheduling with budgeted
  uncertainty}, Discrete Appl. Math. 261 (2019), pp. 93--107,
  \urlprefix\url{https://doi.org/10.1016/j.dam.2018.07.001}.

\bibitem{RobCombOpt}
C. Buchheim and J. Kurtz, \emph{Robust combinatorial optimization under convex
  and discrete cost uncertainty}, EURO J. Comput. Optim. 6 (2018), pp.
  211--238, \urlprefix\url{https://doi.org/10.1007/s13675-018-0103-0}.

\bibitem{QAPlib}
R.E. Burkard, S.E. Karisch, and F. Rendl, \emph{Q{APLIB}---a quadratic
  assignment problem library}, J. Global Optim. 10 (1997), pp. 391--403,
  \urlprefix\url{https://doi.org/10.1023/A:1008293323270}.

\bibitem{Buesing2019}
C. B{\"u}sing, A. B{\"a}rmann, and F. Liers, \emph{Globalized robust
  optimization with $\gamma$-uncertainties}, Tech. {R}ep., Optimization Online,
   2019,
  \urlprefix\url{http://www.optimization-online.org/DB_HTML/2019/06/7253.html}.

\bibitem{BusingBand}
C. B{\"u}sing and F. D'Andreagiovanni, \emph{New Results about Multi-band
  Uncertainty in Robust Optimization}, in \emph{Experimental Algorithms},
  \urlprefix\url{https://doi.org/10.1007/978-3-642-30850-5_7}, Springer Berlin
  Heidelberg, Berlin, Heidelberg, 2012, pp. 63--74.

\bibitem{QAM}
M.J. Feizollahi and I. Averbakh, \emph{The robust (minmax regret) quadratic
  assignment problem with interval flows}, INFORMS J. Comput. 26 (2014), pp.
  321--335, \urlprefix\url{https://doi.org/10.1287/ijoc.2013.0568}.

\bibitem{QAP2015}
M.J. Feizollahi and H. Feyzollahi, \emph{Robust quadratic assignment problem
  with budgeted uncertain flows}, Oper. Res. Perspect. 2 (2015), pp. 114--123,
  \urlprefix\url{https://doi.org/10.1016/j.orp.2015.06.001}.

\bibitem{QAP2012}
M.J. Feizollahi and M. Modarres, \emph{Robust quadratic assignment problem with
  uncertain locations}, Iranian Journal of Operations Research 3 (2012), pp.
  46--65,
  \urlprefix\url{http://iors.ir/journal/browse.php?a_id=323&sid=1&slc_lang=en}.

\bibitem{Fischetti2009}
M. Fischetti and M. Monaci, \emph{Light robustness}, in \emph{Robust and Online
  Large-Scale Optimization: Models and Techniques for Transportation Systems},
  R.K. Ahuja, R.H. M{\"o}hring, and C.D. Zaroliagis, eds., Springer, Berlin,
  Heidelberg,  2009, pp. 61--84,
  \urlprefix\url{https://doi.org/10.1007/978-3-642-05465-5_3}.

\bibitem{survey2014}
V. Gabrel, C. Murat, and A. Thiele, \emph{Recent advances in robust
  optimization: an overview}, European J. Oper. Res. 235 (2014), pp. 471--483,
  \urlprefix\url{https://doi.org/10.1016/j.ejor.2013.09.036}.

\bibitem{Goerigk2021}
M. Goerigk and S. Lendl, \emph{Robust {Combinatorial} {Optimization} with
  {Locally} {Budgeted} {Uncertainty}}, Open Journal of Mathematical
  Optimization 2 (2021), 3,
  \urlprefix\url{https://ojmo.centre-mersenne.org/articles/10.5802/ojmo.5/}.

\bibitem{Gorissen_2013}
B.L. Gorissen and D. den  Hertog, \emph{Robust counterparts of inequalities
  containing sums of maxima of linear functions}, European J. Oper. Res. 227
  (2013), pp. 30--43,
  \urlprefix\url{https://doi.org/10.1016/j.ejor.2012.10.007}.

\bibitem{Gurobi}
 {Gurobi Optimization, LLC}, \emph{Gurobi optimizer reference manual} (2020),
  \urlprefix\url{http://www.gurobi.com}.

\bibitem{VRPSTW}
H. Hashimoto, M. Yagiura, and T. Ibaraki, \emph{An iterated local search
  algorithm for the time-dependent vehicle routing problem with time windows},
  Discrete Optim. 5 (2008), pp. 434--456,
  \urlprefix\url{https://doi.org/10.1016/j.disopt.2007.05.004}.

\bibitem{Kasperski2016}
A. Kasperski and P. Zieli\'{n}ski, \emph{Robust discrete optimization under
  discrete and interval uncertainty: a survey}, in \emph{Robustness analysis in
  decision aiding, optimization, and analytics}, M. Doumpos, C. Zopounidis, and
  E. Grigoroudis, eds., Internat. Ser. Oper. Res. Management Sci., Vol. 241,
  Springer, Cham,  2016, pp. 113--143,
  \urlprefix\url{https://doi.org/10.1007/978-3-319-33121-8_6}.

\bibitem{QAPintro}
T.C. Koopmans and M. Beckmann, \emph{Assignment problems and the location of
  economic activities}, Econometrica 25 (1957), pp. 53--76,
  \urlprefix\url{https://doi.org/10.2307/1907742}.

\bibitem{Kouvelis_1997}
P. Kouvelis and G. Yu, \emph{Robust discrete optimization and its
  applications}, Nonconvex Optimization and its Applications, Vol.~14, Kluwer
  Academic Publishers, Dordrecht, 1997,
  \urlprefix\url{https://doi.org/10.1007/978-1-4757-2620-6}.

\bibitem{Lee2014}
T. Lee and C. Kwon, \emph{A short note on the robust combinatorial optimization
  problems with cardinality constrained uncertainty}, 4OR 12 (2014), pp.
  373--378, \urlprefix\url{https://doi.org/10.1007/s10288-014-0270-7}.

\bibitem{NonlinearRO}
S. Leyffer, M. Menickelly, T. Munson, C. Vanaret, and S.M. Wild, \emph{A survey
  of nonlinear robust optimization}, INFOR Inf. Syst. Oper. Res. 58 (2020), pp.
  342--373, \urlprefix\url{https://doi.org/10.1080/03155986.2020.1730676}.

\bibitem{SurveyVRP}
A. Mor and M.G. Speranza, \emph{Vehicle routing problems over time: a survey},
  4OR 18 (2020), pp. 129--149,
  \urlprefix\url{https://doi.org/10.1007/s10288-020-00433-2}.

\bibitem{Poss12}
M. Poss, \emph{Robust combinatorial optimization with variable budgeted
  uncertainty}, 4OR 11 (2013), pp. 75--92,
  \urlprefix\url{https://doi.org/10.1007/s10288-012-0217-9}.

\bibitem{Poss14}
M. Poss, \emph{Robust combinatorial optimization with variable cost
  uncertainty}, European J. Oper. Res. 237 (2014), pp. 836--845,
  \urlprefix\url{https://doi.org/10.1016/j.ejor.2014.02.060}.

\bibitem{POSS2018}
M. Poss, \emph{Robust combinatorial optimization with knapsack uncertainty},
  Discrete Optim. 27 (2018), pp. 88--102,
  \urlprefix\url{https://doi.org/10.1016/j.disopt.2017.09.004}.

\bibitem{Stochastic}
A. Pr\'{e}kopa, \emph{Stochastic programming}, Mathematics and its
  Applications, Vol. 324, Kluwer Academic Publishers Group, Dordrecht, 1995,
  \urlprefix\url{https://doi.org/10.1007/978-94-017-3087-7}.

\bibitem{Solomon}
M.M. Solomon, \emph{Algorithms for the vehicle routing and scheduling problems
  with time window constraints}, Oper. Res. 35 (1987), pp. 254--265,
  \urlprefix\url{https://doi.org/10.1287/opre.35.2.254}.

\bibitem{Soyster}
A.L. Soyster, \emph{Technical note—convex programming with set-inclusive
  constraints and applications to inexact linear programming}, Operations
  Research 21 (1973), pp. 1154--1157,
  \urlprefix\url{https://doi.org/10.1287/opre.21.5.1154}.

\bibitem{Tadayon2015}
B. Tadayon and J.C. Smith, \emph{Algorithms and complexity analysis for robust
  single-machine scheduling problems}, J. Sched. 18 (2015), pp. 575--592,
  \urlprefix\url{https://doi.org/10.1007/s10951-015-0418-0}.

\bibitem{SurveyAdj}
I. Yan\i{k}o\u{g}lu, B.L. Gorissen, and D. den  Hertog, \emph{A survey of
  adjustable robust optimization}, European J. Oper. Res. 277 (2019), pp.
  799--813, \urlprefix\url{https://doi.org/10.1016/j.ejor.2018.08.031}.

\end{thebibliography}
\newpage
\appendix
%\input{AppendixA.tex}
%\renewcommand\thesection{\Alph{section}}
%\setcounter{section}{0}
%\section{Appendix} \label{A}
\section{Appendix: Numerical study} \label{Sec:App}
The programs were implemented in Python 3.7. To solve the optimization programs, we used Gurobi 9.0.1. \cite{Gurobi} running on machines with Xeon E3-1240 v5 CPUs (4 cores, 3.5 GHz each).
\subsection{Vehicle routing problem with soft time windows under uncertainty}
We elaborate on Example~\ref{Ex:UncertainBounds} based on \cite{Dummies} and \cite{VRPSTW}. We consider a complete digraph $D=(N,A)$ with nodes $N:=[n]$, a start depot $0$ and a copy of the start depot $n+1$, and the digraph $\bar{D} = (V, \bar{A})$ with nodes $V:= [n+1]_0$ and arcs 
\begin{gather*}
\bar{A} = A \cup \{(0,j): \ j \in N\} \cup \{(i,n+1): \ i \in N \cup \{0\}\}.
\end{gather*}
The following data are given: For every arc $a \in \bar{A}$, the travel time is $t_a$ and for every node $i \in N$, a service time $s_i$ and a soft due time $b_i$ is given. Given a homogeneous fleet of $K$ vehicles, all nodes $i \in N$ have to be 'visited' by exactly one vehicle exactly once and with as little delay as possible. The vehicles start and end at the depot. In the following, the binary variables $x_{i,j}^k \in \{0,1\}$ for $(i,j) \in \bar{A}$ and $k \in [K]$ denote whether vehicle $k$ 'uses' arc $(i,j)$ and the real variables $T_i \in \R_{\geq 0}$ for  $i \in V$ denote the arrival time of a vehicle at node $i$. With this notation, we obtain the following optimization program (with $\delta^{\mathrm{out}}(v)$ and $\delta^{\mathrm{in}}(v)$, we denote the outgoing and the incoming arcs of $v \in V$ in the graph $\bar{D}$):
\begin{subequations}
\begin{align}\label{Prob:VRPTSW}
\min_{x,T} & \ \sum_{i \in N} \max \{0, T_i - b_i\}, \\
\text{s.t.} & \ \sum_{k\in[K]} \sum_{(i,j)\in\delta^{\mathrm{out}}(i)}x_{i,j}^{k}=1\ \forall{i}\in{N},  \label{Const:VehR1} \\
	& \sum_{(0,j)\in\delta^{\mathrm{out}}(0)}x_{0,j}^{k}=1 \ \forall{k}\in[K],  \label{Const:VehR2} \\
	& \sum_{(i,j)\in\delta^{\mathrm{in}}(j)}x_{i,j}^{k} - \sum_{(j,i)\in\delta^{\mathrm{out}}(j)}x_{j,i}^{k}=0 \ \forall k\in [K] ,  j\in N,  \label{Const:VehR3} \\
	&\sum_{(i,n+1)\in \delta^{\mathrm{in}}(0)} x_{i,0}^{k} = 1 \ \forall k \in [K],  \label{Const:VehR4} \\
	&x_{i,j}^{k}(T_{i} + s_i + t_{i,j} - T_j)\leq 0 \ \forall k \in [K], \ (i,j)\in A,  \label{Const:VehR5} \\
    &x_{i,j}^{k} \in \{0,1\} \ \forall k \in [K], \ (i,j)\in A, \label{Const:Binary}\\
    &T_i \geq 0 \ \forall i \in V. \label{Const:Time_Non_Negative}
\end{align}
\end{subequations}

%\todo{es sollte gesagt werden dass es verschiedene modelle fuer VRP gibt, und
%wir uns hier lediglich eines davon rausgreifen.}
%\todo{An Frauke: Ist schon dadurch geschehen, dass wir die Spezifizierung nach einer Quelle gemacht haben. Habe noch hinzugefügt, dass Subtour Elimination Constraints durch (59f) redundnt gemacht werden.}

Constraint~\eqref{Const:VehR1} ensures that each $i \in N$ is served exactly once by exactly one vehicle. Constraints~\eqref{Const:VehR2} and \eqref{Const:VehR4} ensure that each vehicle leaves and enters the depot or stays at the depot. In combination with constraints \eqref{Const:VehR1}, \eqref{Const:VehR2} and \eqref{Const:VehR4}, constraint~\eqref{Const:VehR3} ensures that each node is served exactly once and by exactly one vehicle. Constraint~\eqref{Const:VehR5} ensures that, if vehicle $k$ serves node $j$ after node $i$, the arrival time $T_j$ is at least as large as the arrival time $T_i$ added to the time it requires for serving $i$ and going from $i$ to $j$. Finally, \eqref{Const:Binary} and \eqref{Const:Time_Non_Negative} ensure that $x$ is binary and $T$ is non--negative. Note that this formulation is only one of many possibilities to formulate vehicle routing problems -- for an overview, we refer the reader to \cite{SurveyVRP}. We attempt to be robust against scenarios of the set $\bigtimes_{i \in N} [\overline{b}_i - \Delta b_i, \overline{b}_i]$. Solving the $\Gamma$--counterpart for all $\Gamma \in [m]$ would show how many shifts of the due times are possible without any (or only little) delay. \\
For our experiments we use the Solomon instances r101, r102, c101, c102, rc101 and rc102. If these names begin with r, the nodes are generated randomly, if they begin with c, they are clustered, and otherwise some nodes are generated randomly and some are clustered -- for a detailed description of the construction, see \cite{Solomon}. As due time $b_i$ we chose the start time specified in the original instance for the customer, i.e., node $i$. The uncertainty set was constructed randomly, i.e., $\Delta b_i$ is a uniformly distributed random variable in $[0,\overline{b}_i]$. Since we were ultimately aiming to find optimal solutions for the $\Gamma$--counterpart, we tested $N=[8]$ and $N=[10]$, $K \in [3]$ and all $\Gamma \in N$, and calculated the optimal solutions for the respective nominal program. We selected the first $|N|$ customers of the list of customers given in the resp. instance.

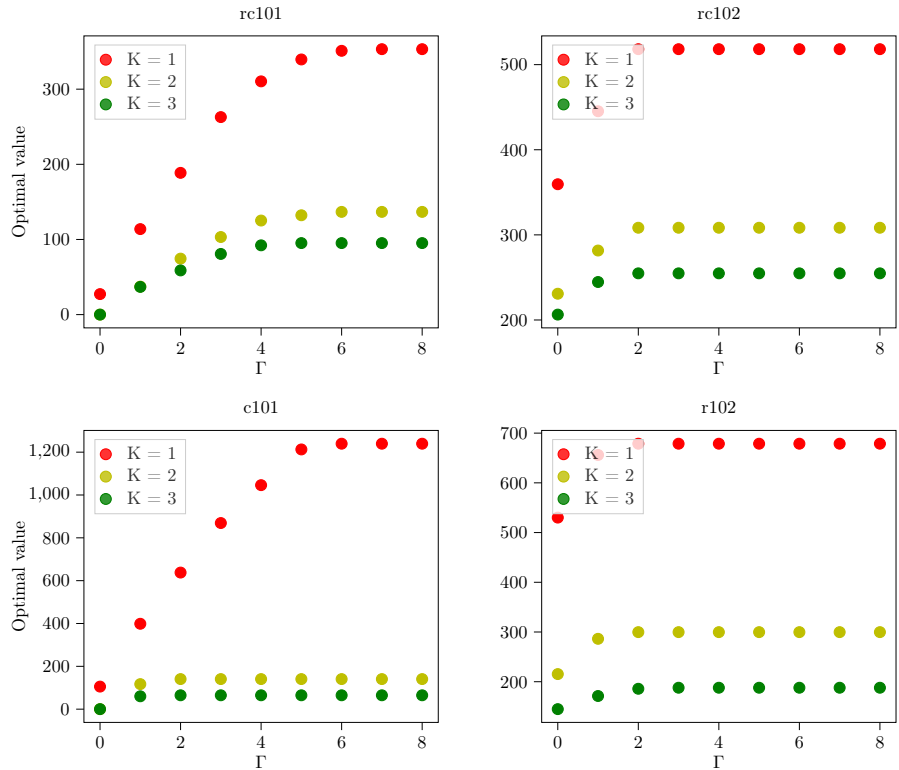
\begin{figure}[H]
\centering
\scalebox{0.68}{% This file was created by tikzplotlib v0.9.1.
\begin{tikzpicture}

\definecolor{color0}{rgb}{0.75,0.75,0}

\begin{groupplot}[group style={group size=2 by 2, horizontal sep=2cm, vertical sep=2cm}]
\nextgroupplot[
legend cell align={left},
legend style={fill opacity=0.8, draw opacity=1, text opacity=1, at={(0.03,0.97)}, anchor=north west, draw=white!80!black},
tick align=outside,
tick pos=left,
title={rc101},
x grid style={white!69.0196078431373!black},
xlabel={\(\displaystyle \Gamma\)},
xmin=-0.4, xmax=8.4,
xtick style={color=black},
y grid style={white!69.0196078431373!black},
ylabel={Optimal value},
ymin=-17.6617143043621, ymax=370.896000391605,
ytick style={color=black}
]
\addplot [semithick, red, mark=*, mark size=3, mark options={solid}, only marks]
table {%
0 27.3005762384212
1 113.78908300188
2 188.637797568131
3 262.840269195654
4 310.401100033197
5 339.539646654361
6 351.043877375048
7 353.234286087243
8 353.234286087243
};
\addlegendentry{K = 1}
\addplot [semithick, color0, mark=*, mark size=3, mark options={solid}, only marks]
table {%
0 0
1 37.2420062992207
2 74.4840125984414
3 103.242358620855
4 125.169771885691
5 132.192959915081
6 136.695240717565
7 136.695240717565
8 136.695240717565
};
\addlegendentry{K = 2}
\addplot [semithick, green!50!black, mark=*, mark size=3, mark options={solid}, only marks]
table {%
0 0
1 36.8796017829337
2 58.8070150477695
3 80.7344283126053
4 92.1990452858838
5 95.1820536748183
6 95.1820536748183
7 95.1820536748183
8 95.1820536748183
};
\addlegendentry{K = 3}

\nextgroupplot[
legend cell align={left},
legend style={fill opacity=0.8, draw opacity=1, text opacity=1, at={(0.03,0.97)}, anchor=north west, draw=white!80!black},
tick align=outside,
tick pos=left,
title={rc102},
x grid style={white!69.0196078431373!black},
xlabel={\(\displaystyle \Gamma\)},
xmin=-0.4, xmax=8.4,
xtick style={color=black},
y grid style={white!69.0196078431373!black},
ymin=190.730210584368, ymax=533.767097387546,
ytick style={color=black}
]
\addplot [semithick, red, mark=*, mark size=3, mark options={solid}, only marks]
table {%
0 359.495947699551
1 445.23821409401
2 518.174511623765
3 518.174511623765
4 518.174511623765
5 518.174511623765
6 518.174511623765
7 518.174511623765
8 518.174511623765
};
\addlegendentry{K = 1}
\addplot [semithick, color0, mark=*, mark size=3, mark options={solid}, only marks]
table {%
0 230.848283790956
1 281.597335318944
2 308.367664933213
3 308.367664933213
4 308.367664933213
5 308.367664933213
6 308.367664933213
7 308.367664933213
8 308.367664933213
};
\addlegendentry{K = 2}
\addplot [semithick, green!50!black, mark=*, mark size=3, mark options={solid}, only marks]
table {%
0 206.322796348149
1 244.621333304128
2 254.808465512234
3 254.808465512234
4 254.808465512234
5 254.808465512234
6 254.808465512234
7 254.808465512234
8 254.808465512234
};
\addlegendentry{K = 3}

\nextgroupplot[
legend cell align={left},
legend style={fill opacity=0.8, draw opacity=1, text opacity=1, at={(0.03,0.97)}, anchor=north west, draw=white!80!black},
tick align=outside,
tick pos=left,
title={c101},
x grid style={white!69.0196078431373!black},
xlabel={\(\displaystyle \Gamma\)},
xmin=-0.4, xmax=8.4,
xtick style={color=black},
y grid style={white!69.0196078431373!black},
ylabel={Optimal value},
ymin=-61.8342067696535, ymax=1301.43875307199,
ytick style={color=black}
]
\addplot [semithick, red, mark=*, mark size=3, mark options={solid}, only marks]
table {%
0 105.359410926432
1 398.712955728297
2 637.767518087942
3 869.096177375391
4 1045.9401565163
5 1212.73382657971
6 1239.47180035191
7 1239.47180035191
8 1239.47180035191
};
\addlegendentry{K = 1}
\addplot [semithick, color0, mark=*, mark size=3, mark options={solid}, only marks]
table {%
0 0.132745950421224
1 116.948821604184
2 140.47075520196
3 140.47075520196
4 140.47075520196
5 140.47075520196
6 140.47075520196
7 140.47075520196
8 140.47075520196
};
\addlegendentry{K = 2}
\addplot [semithick, green!50!black, mark=*, mark size=3, mark options={solid}, only marks]
table {%
0 0.132745950421556
1 60.1160870855148
2 64.7874258412394
3 64.7874258412394
4 64.7874258412394
5 64.7874258412394
6 64.7874258412394
7 64.7874258412394
8 64.7874258412394
};
\addlegendentry{K = 3}

\nextgroupplot[
legend cell align={left},
legend style={fill opacity=0.8, draw opacity=1, text opacity=1, at={(0.03,0.97)}, anchor=north west, draw=white!80!black},
tick align=outside,
tick pos=left,
title={r102},
x grid style={white!69.0196078431373!black},
xlabel={\(\displaystyle \Gamma\)},
xmin=-0.4, xmax=8.4,
xtick style={color=black},
y grid style={white!69.0196078431373!black},
ymin=118.177387116712, ymax=705.412373735494,
ytick style={color=black}
]
\addplot [semithick, red, mark=*, mark size=3, mark options={solid}, only marks]
table {%
0 530.089474429335
1 656.071648523551
2 678.719874343731
3 678.719874343731
4 678.719874343731
5 678.719874343731
6 678.719874343731
7 678.719874343731
8 678.719874343731
};
\addlegendentry{K = 1}
\addplot [semithick, color0, mark=*, mark size=3, mark options={solid}, only marks]
table {%
0 215.35800875793
1 286.29035020091
2 299.734635113224
3 299.734635113224
4 299.734635113224
5 299.734635113224
6 299.734635113224
7 299.734635113224
8 299.734635113224
};
\addlegendentry{K = 2}
\addplot [semithick, green!50!black, mark=*, mark size=3, mark options={solid}, only marks]
table {%
0 144.869886508475
1 171.271169711741
2 185.81049355508
3 187.900882469805
4 187.900882469805
5 187.900882469805
6 187.900882469805
7 187.900882469805
8 187.900882469805
};
\addlegendentry{K = 3}
\end{groupplot}

\end{tikzpicture}}
\caption{Optimal values for the respective $\Gamma$--counterparts for instances rc101 (upper left), rc102 (upper right), c101 (lower left) and r102 (lower right), with $N = [8]$ and $K \in [3]$. If a yellow point for a value of $\Gamma$ is 'missing', its value coincides with the green point of the same $\Gamma$.\label{Fig:OptValues8}}
\end{figure}
In Figure~\ref{Fig:OptValues8} we have the robust optimal values
for $N = [8]$, $K \in [3]$ and the instances rc101, rc102, c101 and
r102 (we have neglected the other two cases and the results for $N=[10]$ because the graphs are similar). As expected, the optimum value, i.e., the waiting time, increases with an
increasing number of vehicles $K$. In addition, at $K = 1$ the optimal value for increasing $\Gamma$ strongly rises, while at $K = 2,3$ the change in the optimal value is not so marked. This is also to be expected: If there is exactly one vehicle, the changes in the due times are supposed to be met by this one vehicle, which is clearly not really possible, especially in clustered settings. However, the total delays are more robust for $K = 2,3$ -- while the robust values differ between $K = 1,2,3$, the difference
between $K=1$ and $K=2$ is much higher than in $K=2$ and $K=3$. So if more vehicles are available, this can lead to more robust solutions. The difference  in the price of robustness is evident, e.g. in c101: For $K = 1$ the nominal optimal value is less than $200$ and for $K = 2,3$ it is $0$. For $\Gamma = 1$ and $K = 1$ we obtain a delay of at least $400$, while for $K = 2,3$ we remain around the optimal nominal value for $K = 1$. For $\Gamma = 2$ the delay increases only slightly and does not change afterwards. However, for $K = 1$, the optimum value increases up to $\Gamma = 6$ and is above $1200$, while for $K = 2,3$ the optimum value is below $200$. We note that for other cases, the difference between the nominal optimal values and the optimal values for $\Gamma \geq 1$ is not as large as can be seen in r102. In this particular case, the increase in nominal optimal values for $\Gamma$ stopped at $\Gamma = 2$ for all $K = 1,2,3$. \\
Table~\ref{Table:8_Patients} and Table~\ref{Table:10_Patients} show the running time to solve
the $\Gamma$--counterpart for $\Gamma \in [2]$, the nominal program for all instances with $N = [8], [10]$ and $K \in [3]$. As the number of
constraints increases with more customers and more vehicles, i.e., rising $|N|$ and $K$, the running time increases in most cases. Note that when reformulating the
$\Gamma$--counterpart, only the objective of the subproblems~\eqref{Prob:SubProbVRPRec} will be affected, while the
optimal solutions of the other subproblems can be reused. Thus, of the
$2|N| + 1$ programs, only $|N|$ programs need to be solved to obtain an optimal solution of the $\Gamma$ counterpart when different values of $\Gamma$ are considered. This explains the fact that the running time for $\Gamma = 2$ is usually at most half as large as that for $\Gamma =1$. We note that the value of $\Gamma$ does not have any other significant influence on the running time and that the running times are relatively high, especially for $|N| = 10$.

\begin{table}[H]
\centering
\caption{Running time of various instances in seconds for $\Gamma= 1,2$, the nominal case, $N = [8]$ and $K = 1,2,3$. \label{Table:8_Patients}}
\begin{tabular}{|l|r|r|r|r|r|r|r|r|r|}
\hline
Instances: & \multicolumn{3}{|c|}{Nominal case}& \multicolumn{3}{|c|}{$\Gamma = 1$} & \multicolumn{3}{|c|}{$\Gamma = 2$} \\
  \cline{2-10}
$N = [8]$ & $K = 1$ & $K = 2$ & $K = 3$ & $K = 1$ & $K = 2$ & $K = 3$ & $K = 1$ & $K = 2$ & $K = 3$\\ \hline
r101 & 17 & 1 & 2 & 79 & 43 & 26& 40 & 15 & 14 \\ \hline
r102 & 11& 22 & 15 & 105 & 356 &246 & 49 & 136 & 123 \\ \hline 
c101 & 1 & 1 & 1 & 28 & 7 & 8 & 19 & 4 & 5 \\ \hline
c102 & 10 & 21 & 24 & 124 & 394 & 434 & 62 & 183 & 298 \\ \hline
rc101 & 5 & 1 & 2 & 107 & 156 & 224& 30& 19& 22 \\ \hline
rc102 & 7 & 69 & 291 & 101 & 823 & 2965 & 49 & 284& 1282 \\ \hline
\end{tabular}
\end{table}

\begin{table}[H]
\centering
\caption{Running time of various instances in seconds for $\Gamma= 1,2$, the nominal case, $N = [10]$ and $K = 1,2,3$. If no optimal solution has been obtained after 24 hours, the resp. fields are marked with --. \label{Table:10_Patients}}
\begin{tabular}{|l|r|r|r|r|r|r|r|r|r|}

\hline
Instances: & \multicolumn{3}{|c|}{Nominal case}& \multicolumn{3}{|c|}{$\Gamma = 1$} & \multicolumn{3}{|c|}{$\Gamma = 2$} \\
  \cline{2-10}
$N = [10]$ & $K = 1$ & $K = 2$ & $K = 3$ & $K = 1$ & $K = 2$ & $K = 3$ & $K = 1$ & $K = 2$ & $K = 3$ \\ \hline
r101 & 551 & 6 & 3 & 9080 & 5178 & 2316 & 5612 & 1525 & 376 \\ \hline
r102 & 1175 & 1066 & 119 & 15358 & 54885 & 15637 & 9433 & 18285 & 5875 \\ \hline 
c101 & 15 & 2 & 5 & 8026 & 2444 & 168 & 4161 & 65 & 28 \\ \hline
c102 & 1566 & 431 & 111 & 22081 & 32098 & 21335 & 15285 & 12692 & 12703  \\ \hline
rc101 & 681 &14 &  7& 8279 & 11024 & 9364 & 2949 & 1253& 350 \\ \hline
rc102 & 902 & 2505 & 8425 & 13327 & 70041 & -- & 6804 & 37788 &--  \\ \hline

\end{tabular}
\end{table}
%\todo{Für die leeren Felder hat der HPC den Job beendet, ohne dass die Walltime gesprengt wurde (aus Speicherplatzgründen. Ich versuchs heut nochmal, ansonsten könnte man unter Umständen auch versuchen, mit einer Gap zu arbeiten. Oder die Laufzeiten wegzulassen - dass ein VRPSTW auf Optimalität zu lösen praktisch ein Ding der Unmöglichkeit ist, ist an sich ja kein Staatsgehemnis... :)}

This concludes our numerical study of the VRPGTW under uncertainty. As already mentioned, we used an optimization oracle to solve the programs given in Example~\ref{Ex:UncertainBounds} as a MINLP instead of using any VRPTGW solvers to demonstrate that our reformulation can be solved to optimality. In the future, it might be interesting to conduct experiments including instances with more customers but rather than solving them to global optimality, they could be solved only to a certain gap, i.e., to find solutions which are 'sufficiently robust' or to apply a VRPTGW oracle.

\subsection{Quadratic assignment problem under uncertainty}
Here, we solve and compare different reformulations of the $\Gamma$--counterpart of the QAP. We have chosen instances from \cite{QAP2015} and from the QAPLIB \cite{QAPlib}. The goal of this section is to prototypically evaluate whether the new reformulations can be solved within a similar order of
magnitude when compared to that of the nominal versions. As we do not have an efficient problem--specific QAP oracle at hand, we chose small instances where the $\Gamma$--counterpart could be solved with Gurobi within 24 hours. As expected, instances with less uncertain coefficients are computationally easier to handle. Therefore, by choosing scr12, we included an instance with $c_{i,j} = 0$ for some $(i,j) \in [12]^2$. We also chose fei9, an instance that was examined in \cite{QAP2015}. For fei9, the number of facilities is $n = 9$, while for scr12, it is %(from \cite{QAPlib})
$n =12$ (both taken from \cite{QAP2015}). Finally, we also chose nug12 from \cite{QAPlib}. 
%% and omitted the last
%six facilities and locations, since no other instance with $n<12$ is available
%in QAPLIB. %\todo{Kann man hier nicht sagen, dass man die letzten $6$
          %Facilities und locations weglässt? Käme mir einfacher
%vor}.
%The modified instance with $n=6$ is denoted as nug6. 
%\todo{Matrizen? $u^k$ waren bisher maximal Vektoren, wieso sind das hier Matrizen?} Each instance describes a QAP, i.e., the distances between the locations and the amount of transferred goods between the facilities are given. For every instance we chose three matrices $u^k$ with $k \in [3]$ which define the maximum deviation for the flow $\Delta c_{ij}$ for $(i,j) \in [n]$.
For each instance, we generated three different uncertainty sets. For fei9, the uncertainty set $\mathcal{U}_1$ is taken from \cite{QAP2015}. Other uncertainty sets, denoted by $\mathcal{U}_2$ and
$\mathcal{U}_3$, are generated randomly: for all $(i,j) \in [n]^2$, $\Delta c_{ij} \in [0,\bar{c}_{ij}]$ is randomly chosen. %each tuple $(i,j)\in [n]^2$. %\todo{Hier gibt es einen Notationsclash: die uncertainty sets waren immer $\mathcal{U}$, aber es haben sich $\mathcal{U}^i$ immer auf einzelne Funktionen $f^i$ bezogen - ich weiß nicht, ob das hier genauso ist... Was ich später verwirrend finde, ist der Begriff uncertainty setting, den gabs im Paper vorher gar nicht und ich kann mich auch nicht entsinnen, den schon mal gelesen zu haben. Kann man nicht in Formeln beschreiben, wie das geht? Also an den Instanznamen irgendwie einen Index dran hängen, der das Setting klar macht und sagen, wie im jeweiligen Setting verfahren wird..?}
%matrices with entries $u^k_{ij} = \Delta c_{ij} \in [0,\bar{c}_{ij}],
%(i,j)\in [n]^2$ and $k \in \{2,3\}$.
%\todo{ich wollte hier gerade Formeln vermeiden, eben weil die Notation
%nicht ganz klar ist, und ich dachte es ist leicht genug dass wir es
%einfach in Text schreiben. ansonsten muss man als leser nochmal
%zurueckblaettern... $u^1$ ist nicht die unsicherheitsmenge,
%sondern die Instanz, die dann so heisst (so ist es ja im Plot, die
%Kurven haben das Label).}
For scr12 and nug6, $\mathcal{U}_1$ is generated by setting $\Delta c_{ij} =
0.1  \bar{c}_{ij}$ for all $(i,j) \in [n]^2$. Furthermore, $\mathcal{U}_2$ and $\mathcal{U}_3$ are generated randomly analogously to fei9. 
%, equals the number of $\bar{c}_{ij} \neq 0$.

%Starting with the first experiment we have to compute the robust optimal solution for different $\Gamma$.
In Figure~\ref{Fig:obj} the change in the objective value for different $\Gamma$ can be observed for two of our instances are shown. As expected, the optimal objective value is increasing in $\Gamma$. As can be seen for scr12, only a mild increase in cost of
robust protection can be seen for increasing values of $\Gamma$. %is
%almost negligible, even for large $\Gamma$ values it gets
%remarkable for larger uncertainties. %\todo{Dennis: Den letzten Satz
                                %hat die Lektorin nicht
%verstanden. Und ich auch nicht...}

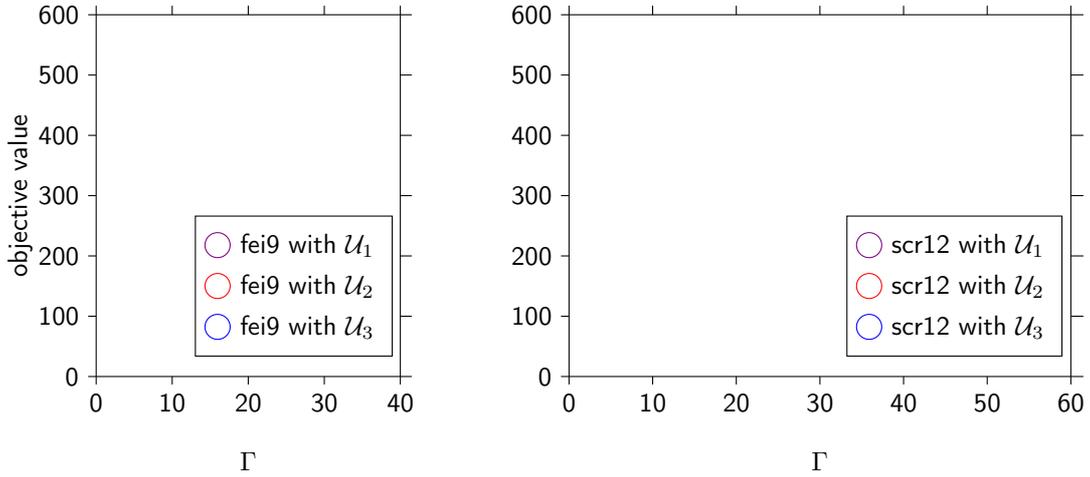
\begin{figure}[H]
\begin{center}
\centering
\begin{subfigure}[c]{0.43\textwidth}
\begin{tikzpicture}[y=.008cm, x=.1cm,font=\sffamily, only marks]
 	%axis
	\draw (0,0) -- coordinate (x axis mid) (40,0);
	\draw (0,600) -- coordinate (x axis mid) (40,600);
	\draw (40,0) -- coordinate (y axis mid) (40,600);
    	\draw (0,0) -- coordinate (y axis mid) (0,600);
    	
    	%ticks
    	\foreach \x in {0,10,...,40}
     		\draw (\x,0pt) -- (\x,-3pt)
     		node[anchor=north] {\x};
     	\foreach \x in {0,10,...,40}
     		\draw (\x,615) -- (\x,600);
    	\foreach \y in {0,100,...,600}
     		\draw (0pt,\y) -- (-3pt,\y) 
     			node[anchor=east] {\y }; 
     	\foreach \y in {0,100,...,600}
     		\draw (40,\y) -- (41.5,\y);
	%labels      
	\node[below=0.8cm] at (20,-10) {$\Gamma$};
	\node [rotate=90] at(-10,300) {objective value};

	%plots
	\draw[blue] plot[mark=*, style={mark size=0.8pt}, mark options={fill=blue}] 
	    file {neuobjpap9u2.data};
	\draw[violet] plot[ mark=*, style={mark size=0.8pt}, mark options={fill=violet}] 
	    file {objpap9.data};
	 \draw[red] plot[ mark=*, style={mark size=0.8pt}, mark options={fill=red}] 
	    file {neuobjpap9u1.data};
	%\draw[black] plot[ mark=none, mark options={fill=white}] 
	%   file {objscr15.data};	    
	 \matrix [draw,left] at (39,150) {
  \node [shape=circle, draw=violet,label=right:fei9 with $\mathcal{U}_1$ ] {}; \\
   \node [shape=circle, draw=red,label=right:fei9 with $\mathcal{U}_2$ ] {}; \\  	 
  \node [shape=circle, draw=blue,label=right: fei9 with $\mathcal{U}_3$ ] {}; \\
  %\node [shape=circle, draw=black, label=right: scr15 in $10^2$] {}; \\
};
\end{tikzpicture}
\subcaption{Optimal objective value $\times 10^{-4}$ for different values of $\Gamma$  for fei9.  \label{fig:objpap9}}
\end{subfigure}%
\hspace{0.1cm}
\begin{subfigure}[c]{0.55\textwidth}
\begin{tikzpicture}[y=.008cm, x=.11cm,font=\sffamily, only marks]
 	%axis
	\draw (0,0) -- coordinate (x axis mid) (60,0);
		\draw (0,600) -- coordinate (x axis mid) (60,600);
    	\draw (60,0) -- coordinate (y axis mid) (60,600);
    	\draw (0,0) -- coordinate (y axis mid) (0,600);
    	%ticks
    	\foreach \x in {0,10,...,60}
     		\draw (\x,0pt) -- (\x,-3pt)
			node[anchor=north] {\x};
		\foreach \x in {0,10,...,60}
     		\draw (\x,615) -- (\x,600);
    	\foreach \y in {0,100,...,600}
     		\draw (0pt,\y) -- (-3pt,\y) 
     			node[anchor=east] {\y }; 
     	\foreach \y in {0,100,...,600}
     		\draw (60,\y) -- (61.5,\y);
	%labels      
	\node[below=0.8cm] at (30,-10) {$\Gamma$};

	%plots
	\draw[blue] plot[ mark=*, style={mark size=0.8pt}, mark options={fill=blue}] 
	    file {neuobjscr12u2.data};
	\draw[violet] plot[ mark=*, style={mark size=0.8pt}, mark options={fill=violet}] 
	    file {objscr12.data};
	 \draw[red] plot[ mark=*, style={mark size=0.8pt}, mark options={fill=red}] 
	    file {neuobjscr12u1.data};
	%\draw[black] plot[ mark=none, mark options={fill=white}] 
	%   file {objscr15.data};	    
	 \matrix [draw,left] at (59,150) {
  \node [shape=circle, draw=violet,label=right:scr12 with $\mathcal{U}_1$] {}; \\ 
   \node [shape=circle, draw=red,label=right:scr12 with $\mathcal{U}_2$] {}; \\
  \node [shape=circle, draw=blue,label=right: scr12 with $\mathcal{U}_3$] {}; \\ 
  %\node [shape=circle, draw=black, label=right: scr15 in $10^2$] {}; \\
};
\end{tikzpicture}
\subcaption{Optimal objective value $\times 10^{-2}$ for different values of $\Gamma$ for scr12. \label{fig:objscr12}}
\end{subfigure}
\caption{Optimal solution for different instances.}\label{Fig:obj}
\end{center}
\end{figure}

Now we compare the running time of different equivalent formulation of the $\Gamma$--counterpart. In particular, we test following formulations: 
\begin{itemize}
\item \texttt{QAP}: Formulation~\eqref{Prob:QAP_Reform_Two}.
\item $\texttt{QAP}_\texttt{red}$: Formulation~\eqref{Prob:QAP_Reform_Two} after reducing the number of subproblems with Theorem~\ref{Theo:RedNumber}.
\item \texttt{MIP}: A linearized QAP under uncertainty after applying Theorem~1 of \cite{IntroGamma}.
\item \texttt{BP}: A linearized QAP under uncertainty after applying Proposition~\ref{Prop:Tractable_Gamma_Comb}.
\item $\texttt{BP}_\texttt{red}:$ A linearized QAP under uncertainty after applying Proposition~\ref{Prop:Tractable_Gamma_Comb} and Theorem 1 of \cite{Lee2014}. 
\end{itemize} 
In particular, we apply a standard linearization: the product of two binary variables $x$ and $y$ can be replaced by a binary variable $z$ and the set of inequalities
\begin{gather*}
z \leq x, \ z \leq y, \ z \geq x + y - 1.
\end{gather*}
The nominal programs can be solved  within a few seconds.
A comparison of running times for fei9, scr12 and nug12 and $\Gamma = 1$ can be found in
Tables \ref{Table_U1}, \ref{Table_U2} and \ref{Table_U3}. If no optimal solution could be computed after 24 hours, we stopped the process. Running times are measured in seconds. 
\begin{table}[H]
\centering
\caption{Comparison of running times for different instances for $\Gamma = 1$ and the deterministically constructed uncertainty set $\mathcal{U}_1$. -- if not solvable within 24 hours. \label{Table_U1}} 
\begin{tabular}{|l|r|r|r|}
\hline 
 CPU (s)& nug12 & fei9 & scr12  \\ 
\hline 
$\texttt{QAP}$& -- & 3420 & 36579
 \\ 
\hline 
$\texttt{QAP}_\texttt{red}$ & 1054  & 174 & 221  \\ 
\hline 
$\texttt{MIP}$ & 31417 & 25 &  1718 \\ 
\hline 
$\texttt{BP}$ & -- & 86243 & -- \\ 
\hline 
$\texttt{BP}_\texttt{red}$ & 34318 & 4096 & 49126 \\ 
\hline
\end{tabular}
\end{table}
\begin{table}[H]
\centering
\caption{Comparison of running times for different instances for $\Gamma = 1$ and the deterministically constructed uncertainty set $\mathcal{U}_2$. -- if not solvable within 24 hours. \label{Table_U2}} 
\begin{tabular}{|l|r|r|r|}
\hline 
 CPU (s)& nug12 & fei9 & scr12  \\ 
\hline 
$\texttt{QAP}$& -- & 3542 & 74641
 \\ 
\hline 
$\texttt{QAP}_\texttt{red}$ & 591  & 178 & 298  \\ 
\hline 
$\texttt{MIP}$ & 24655 & 25 &  819 \\ 
\hline 
$\texttt{BP}$ & -- & -- & -- \\ 
\hline 
$\texttt{BP}_\texttt{red}$ & -- & 4538 & 73847 \\ 
\hline
\end{tabular}
\end{table}
\begin{table}[H]
\centering
\caption{Comparison of running times for different instances for $\Gamma = 1$ and the randomly constructed uncertainty set $\mathcal{U}_3$. - if not solvable within 24 hours. \label{Table_U3}} 
\begin{tabular}{|l|r|r|r|}
\hline 
 CPU (s)& nug12 & fei9 & scr12  \\ 
\hline 
$\texttt{QAP}$ & -- & 3503 & 45851
 \\ 
\hline 
$\texttt{QAP}_\texttt{red}$ & 9507  & 178 & 275  \\ 
\hline 
$\texttt{MIP}$ & 19550 & 30 &  1860 \\ 
\hline 
$\texttt{BP}$ & -- & -- & -- \\ 
\hline 
$\texttt{BP}_\texttt{red}$ & -- & 4384 & 56979 \\ 
\hline
\end{tabular}
\end{table}
It is evident that the instances with $n=12$ can be solved more
efficiently than the linearizations after reducing the number of
subproblems by excluding all redundant scenarios (applying
Theorem~\ref{Theo:RedNumber}, neglecting identical subproblems and
taking symmetry of coefficients into account), for all regarded
uncertainty sets. Only for the smaller instance, \texttt{MIP} is
faster. This demonstrates the benefit of the reformulations
proposed here. Without using them, the corresponding robust
counterparts are algorithmically very challenging. All instances have in common that without reducing
the number of programs, i.e., avoiding a repetition of scenarios or
applying Theorem~\ref{Theo:RedNumber}, these instances cannot be
solved within the time limit, even for smaller instances.

Finally, we would like to point out two things: Firstly, if one would
like to solve $\Gamma$--counterpart for different values of $\Gamma$,
it is preferable to apply $\texttt{QAP}_\texttt{red}$ since one only
has to calculate the optimal solutions of the subproblems for $\Gamma
= 1, 2$, since the value of $\Gamma$ does not influence the
subproblems. Secondly, this computational study 
demonstrates that our formulations are applicable in
practice. Naturally, instead of using Gurobi, one can also use
algorithms that solve QAPs more efficiently. However, for
our purposes, our method proved to be highly beneficial, when compared to the standard
linearization approach.

\section{Appendix: Uncertainty in the Constraints}\label{sec:UncConst}
Here, we consider the case of a constraint being subject to uncertainty, i.e., program
\begin{align}\label{Prob:NomUncertainConstraint}
\begin{split}
\inf_{x \in \mathcal{X}} & \ f(x), \\
\mathrm{s.t.} & \sum_{i \in [m]} \overline{g}_i(x) \leq 0.
\end{split}
\end{align}

Analogously to Section~\ref{sec:Non_Linear_Objective_Functions}, we assume that the functions $\overline{g}_i$ are subject to  uncertainty, i.e., we set 
\begin{gather*}
g_i \colon \mathcal{X} \times \mathcal{U}_i \to \R, 
\end{gather*}
with $g_i(x, \overline{u}^i) := \overline{g}_i(x)$ for a nominal scenario $\overline{u}^i \in \mathcal{U}_i$. Thus, program~\eqref{Prob:NomUncertainConstraint} under uncertainty can be stated as  
%For a MIP, the robust counterpart is given in Problem~\eqref{Prob:RobCounterpartMIP}.  For nonlinear constraints, the same transformation can be performed as in Lemma~\ref{Lem:ReModel}. The new problem is
%The non-linear optimization problem with uncertain constraints can be written as:
\begin{align}\label{Prob:uncertainconstrains}
\begin{split}
\inf_{x \in \mathcal{X}} & \ f(x),\\
\mathrm{s.t.}& \sum_{i\in [m]} \sup_{u^i \in \mathcal{U}^i}g_i(x,u^i) \le 0.
\end{split}
\end{align}
The $\Gamma$--counterpart of program~\eqref{Prob:uncertainconstrains} is given by
\begin{align}\label{Prob:GammaCounterpartuncertainconstrains}
\begin{split}
\inf_{x \in \mathcal{X}} & \ f(x),\\
\text{s.t.}& \sup_{\mathcal{S} \subseteq [m]: |\mathcal{S}| \leq \Gamma} \left\{ \sum_{i \in \mathcal{S}} \sup_{u^i \in \mathcal{U}_i} g_i(x,u^i) + \sum_{i \in [m] \setminus \mathcal{S}} g_i(x, \overline{u}^i) \right\}  \leq 0.
\end{split}
\end{align}

Equivalent to Lemma~\ref{Lem:ReModel}, we can obtain a reformulation without the outer supremum operator:
\begin{lemma}\label{Lem:ConstraintModel}
If $\Gamma \in [m]$, then program~\eqref{Prob:GammaCounterpartuncertainconstrains} is equivalent to 
\begin{align}\label{Prob:uncertainconstrainsdual}
\begin{split}
\inf_{x,p,\theta} & \ f(x),\\
\mathrm{s.t.} &\ \sum_{i \in [m]} g_i(x,\bar{u}^i)+p_i + \theta \Gamma \le 0, \\
&p_i +\theta \ge \sup_{u^i \in \mathcal{U}_i} g_i(x,u^i)-g_i(x,\bar{u}^i) \ \forall i \in [m],\\
&p,\theta \ge 0, \\
&x \in \mathcal{X}.
\end{split}
\end{align}
\end{lemma}
\begin{proof}
The proof is almost identical to the proof of Lemma~\ref{Lem:ReModel} since the constraint subject to uncertainty of $\Gamma$--counterpart~\eqref{Prob:GammaCounterpartuncertainconstrains} is the objective of the $\Gamma$--counterpart of the $\Gamma$--counterpart~\eqref{Prob:GeneralGammaCounterpart} with uncertainty in the objective. In this case, we obtain
\begin{alignat*}{2}
0 \geq & \inf_{\theta,p} \ \sum_{i \in [m]} g_i(x,\bar{u}^i)+p_i+ \theta \Gamma  ,\\
& \ \text{s.t.}  \ p_i +\theta \ge \sup_{u^i \in \mathcal{U}_i} 
g_i(x,u^i)-g_i(x,\bar{u}^i),\ \forall i \in [m],\\
 & \ \ \ \ \ \ p,\theta \geq 0.
\end{alignat*}
Thus, the $\inf$ operator can be omitted and the claim is proven.
%\begin{align}
%\min_x & \ f(x),\\
%\text{s.t.} & \max_{ \mathcal{S} \subseteqeq [m], |\mathcal{S}| \le \Gamma} \sum_{ i \in \mathcal{S}} \max_{u^i \in \mathcal{U}_i} g_i(x,u^i) + \sum_{i\in [m]\setminus \mathcal{S}} g_i(x,\bar{u}^i)\le 0. \label{Eq:constrain}
%\end{align}
%We define for all $i \in [m]$ a variable $s_i$ that states whether $i \in \mathcal{S}$\\
%\begin{align*}
%s_i = \left\{
%\begin{array}{ll}
%1 & \text{if } i \in \mathcal{S} \\
%0 & \, \text{otherwise.} \\
%\end{array}
%\right.
%\end{align*}
%With that we can reformulate the Equation~\eqref{Eq:constrain} to
%\begin{align*}
%\max_{s}& \sum_{i \in [m]} g_i(x,\bar{u}^i) + \sum_{i\in [m]} s_i(\max_{u^i \in \mathcal{U}_i} g_i(x,u^i) - g_i(x, \bar{u}^i))\le 0,\\
%& \sum_{i \in [m]} s_i \le \Gamma, \\
%& s_i \in \{0,1\}, \qquad \forall i \in [m].
%\end{align*}
%We now can recycle the proof of Lemma~\ref{Lem:ReModel} to obtain 
%\begin{align*}
%\begin{split}
%\min_{\theta,p}& \sum_{i \in [m]} g_i(x,\bar{u}^i)+p_i + \theta \Gamma \le 0 ,\\
%&p_i +\theta \ge \max_{u^i \in \mathcal{U}_i} \{g_i(x,u^i)-g_i(x,\bar{u}^i)\}, \qquad \forall i \in [m],\\
%&p,\theta \ge 0.
%\end{split}
%\end{align*}
%Because of the direction of the inequality the minimum can be omitted. 
%This leads to \eqref{Prob:uncertainconstrainsdual}.
\end{proof}
Comparing the reformulations of the $\Gamma$--counterparts of Lemmas~\ref{Lem:ReModel} and \ref{Lem:ConstraintModel}, the only difference is the inequality 
\begin{gather}\label{Ineq:Difference}
\sum_{i \in [m]} g_i(x,\overline{u}^i) + p_i + \Gamma \theta \leq 0.
\end{gather} 
However, the left hand side of inequality~\eqref{Ineq:Difference} is the objective of program~\eqref{Prob:GeneralReModel}. More importantly, the bottleneck of both $\Gamma$--counterparts is the inequality %in both cases, one needs to obtain a tractable formula of
% In Proposition~\ref{Prop:Gamma_MIP} we determined the robust counterpart for  uncertain constraints of a MIP. If we compare that with problem \eqref{Prob:uncertainconstrains}, the following similarities and differences arise. $\sum_{i\in [m]}g_i(x,u^i)$ equals one uncertainty \[ \sum_{j \in [n]} \overline{a}_{i,j} x_j + \sum_{j \in [n]} \Delta a_{i,j} |x_{ij}| \]
%
%%In Proposition~\ref{Prop:Gamma_MIP} we determined the robust counterpart for uncertain constraints which are given in a MIP. If we compare that with problem \eqref{Prob:uncertainconstrains}, the following similarities and differences arise. $\sum_{i\in [m]}g_i(x,u^i)$ equals one uncertainty \[ \sum_{j \in [n]} \overline{a}_{i,j} x_j + \sum_{j \in [n]} \Delta a_{i,j} |x_{ij}| \]
%
%which is one row of the equation system $Ax \le b$. In this section $g_i$ is not linear and we do not demand $x$ to be positive. We therefore do not need to use the absolute value.

%We try to reformulate the problem further, to find a tractable formula for 
\[p_i +\theta \ge \sup_{u^i \in \mathcal{U}_i} g_i(x,u^i)-g_i(x,\bar{u}^i),\]
for which we already discussed several reformulations in Section~\ref{sec:Non_Linear_Objective_Functions}. Hence, the reformulation techniques are still applicable here and for MINLPs, we obtain, under the analogue of Assumption~\ref{Ass:General}, the same reformulations: 
%We had a constraint of the same form in section~\ref{sec:Non_Linear_Objective_Functions}. Therefore we proceed equivalent and show, that the same reformulations we did there for $f_i(x,u^i)$ are also possible for $g_i(x,u^i)$. 
%Analogously to section~\ref{sec:Non_Linear_Objective_Functions} we make the same assumptions we did for $f_i(x,u^i)$ now for $g_i(x,u)$ and $\mathcal{U}_i$.
\begin{theorem}\label{Theo:FenchelConstraint}
Assume that Assumption~\ref{Ass:General} holds for all $g_i(x,u^i)$ and $\mathcal{U}_i$ and $\Gamma \in [m]$. Then program~\eqref{Prob:GammaCounterpartuncertainconstrains} is equivalent to 
\begin{align}\label{Prob:FenchelConstraint}
\begin{split}
\inf_{x,p,\theta} & \ f(x),\\
\text{s.t.} & \sum_{i \in [m]} g_i(x,\bar{u}^i)+p_i + \theta \Gamma \le 0, \\
&p_i + \theta \geq (\overline{u}^i)^Tv^i + \support{i} - g_{i,*}(x,v^i) - g_i(x, \overline{u}^i) \ \forall i \in [m], \\
&p,\theta \ge 0,\\
&x \in \mathcal{X}.
\end{split}
\end{align}
Furthermore, if Assumption~\ref{Ass:GeneralLinear} holds, the program~\eqref{Prob:GammaCounterpartuncertainconstrains} is equivalent to 
\begin{align}\label{Prob:FenchelGammaCounterpartLinearConst}
\begin{split}
\inf_{x,p,\theta} &  \ f(x) \\
\mathrm{s.t.} & \ \sum_{i \in [m]} g_i(x,\bar{u}^i)+p_i + \theta \Gamma \le 0,\\
& p_i + \theta \geq   \supportx{i}  \ \forall i \in [m], \\
& p, \theta \geq 0, \\
& x \in \mathcal{X}. 
\end{split}
\end{align}
\end{theorem}
%\begin{proof}
%%With $\Gamma$ being integral we apply Lemma~\ref{Lem:ConstraintModel}. 
%%Using Proposition~\ref{Prop:GeneralPrinciple} it is trivial to show that for all $i \in [m]$
%%\[(\overline{u}^i)^Tv^i + \support{i} - g_{i,*}(x,v^i) - g_i(x, \overline{u}^i)\]
%%is satisfied if and only if 
%%\[p_i +\theta \ge \max_{u^i \in \mathcal{U}_i} g_i(x,u^i)-g_i(x,\bar{u}^i) \]
%%holds for every $i \in [m]$.
%\end{proof}
The proof is omitted (see  Corollaries~\ref{Prop:FenchelGammaCounterpart} and~\ref{Prop:FenchelGammaCounterpartLinear}). 
However, we would like to point out that, contrary to Section~\ref{sec:Non_Linear_Objective_Functions}, the dual variables $p$ and $\theta$ can not be eliminated due to the additional inequality~\eqref{Ineq:Difference}. This arises as $\theta$ is additionally multiplied with $\Gamma$. However, since the feasible set $\mathcal{X}$ was subject to uncertainty in program~\eqref{Prob:GammaCounterpartuncertainconstrains}, it is not necessary to optimize over $\mathcal{X}$ only.

\end{document}